\documentclass[review]{siamltex}

\usepackage{url}
\usepackage{caption}
\usepackage{xparse}
\usepackage{mathtools}
\usepackage{bbm}
\usepackage{enumitem}
\usepackage{esvect}
\usepackage{centernot}
\usepackage{float} 
\usepackage{fullpage}
\usepackage{amsfonts}
\usepackage{latexsym}
\usepackage{amssymb}
\usepackage{color}
\usepackage{hyperref}
\usepackage{graphicx}
\graphicspath{{./figures/}}
\usepackage{epstopdf}
\usepackage{algorithm,algpseudocode}
\usepackage{verbatim}
\usepackage{mathtools}
\usepackage{enumitem}
\usepackage{esvect}
\usepackage{centernot}
\usepackage{float} 
\usepackage{siunitx}
\usepackage{mathrsfs}
\usepackage[flushleft]{threeparttable}
\usepackage{textcomp}
\usepackage{booktabs}

\allowdisplaybreaks
\def\vgap{\vspace*{.1in}}
\usepackage{algorithm,algpseudocode}

\numberwithin{equation}{section}

\newcommand{\E}{\mathbb{E}}

\newcommand{\R}{\mathbb{R}}

\newcommand{\La}{\mathcal{L}}

\def \grad {\nabla}
\DeclareMathOperator*{\argmax}{arg\,max}
\DeclareMathOperator*{\argmin}{arg\,min}

\newcommand{\inner}[2]{\langle {#1,#2} \rangle}
\newcommand{\norm}[1]{\left\lVert #1 \right\rVert}
\newcommand{\bigO}{\mathcal{O}}

\newcommand{\tsum}{{\textstyle \sum}}

\newcommand{\ep}{\epsilon}
\newcommand{\lam}{\lambda}
\newcommand{\ccp}{\text{p.c.c.}}

\makeatletter
\newcommand{\algmargin}{\the\ALG@thistlm}
\makeatother
\newlength{\whilewidth}
\algdef{SE}[parWHILE]{parWhile}{EndparWhile}[1]
  {\parbox[t]{\dimexpr\linewidth-\algmargin}{%
     \hangindent\whilewidth\strut\algorithmicwhile\ #1\ \algorithmicdo\strut}}{\algorithmicend\ \algorithmicwhile}%
\algnewcommand{\parState}[1]{\State%
  \parbox[t]{\dimexpr\linewidth-\algmargin}{\strut #1\strut}}

\def \hp {\hat{p}}

\def \xstar {x^*}

\def \pikbar {\bar{\pi}_k}
\def \pibar {\bar{\bpi}}
\def \pbar {\bar{p}}

\def \txt {\widetilde{x}^t}

\def \pt {p_t}
\def \pk {p_k}

\def \xt {x_t}
\def \zt {z_t}
\def \pikt {\pi_{k,t}}
\def \pikp {\pi_{k,t+1}}

\def \ptt {p_{t-1}}
\def \xtt {x_{t-1}}

\def \piktt {\pi_{k, t-1}}
\def \tauq {\tau_q}

\def \mup {\mu_p}

\def \gmupi {\mathbf{g}_{\mupi}}
\def \gmupik {g_{\mupi, k}}
 \def \bW {\mathbf{W}}

\def \sumk {\tsum_{k=1}^{K}}
\def \sumt {\tsum_{t=1}^{N}}

\def \minx {\min_{x \in X}}
\def \maxp {\max_{p \in P}}

\def \maxpik {\max_{\bpi \in \bPi}}

\def \argmaxpik {\argmax_{\pik \in \Pik}}

\def \argminx {\argmin_{x \in X}}
\def \argmaxp {\argmax_{p \in P}}

\def \pk {p_k}

\DeclareDocumentCommand{\fk}{O{x}}{f_k(T_K #1)}

\DeclareDocumentCommand{\dscenp}{O{x} O{p} O{\pik}}{(\inner{\Tk #1}{#3} - g_k^*(#3, \xi_k))}

\DeclareDocumentCommand{\dsecon}{O{x} O{p_k} O{\bpi}}{ \sumk #2 \dscenp[#1][#2][#3] - \phi^*(#2)}

\DeclareDocumentCommand{\dseconx}{O{x} O{p} O{\bpi}}{\inner{ #1}{#2 \bT #3}}
\DeclareDocumentCommand{\fir}{O{x} O{p} O{\pik}}{ f_0(#1)}
\DeclareDocumentCommand{\pfunc}{O{x} O{p} O{\pik}}{\fir +  \psecon[#1][#2][#3]}
\DeclareDocumentCommand{\dfunc}{O{x} O{p_k} O{\pik}}{ \fir +  \dsecon[#1][#2][#3]}    
\DeclareDocumentCommand{\pprob}{O{x} O{p} O{\pik}}{\minx \fir + \maxp \psecon[#1][#2][#3]}
\DeclareDocumentCommand{\dprob}{O{x} O{p} O{\pik}}{\minx \fir + \maxp \maxpik \dsecon[#1][#2][#3]}

\newcommand{\lunder}{\underline{\La}}
\newcommand{\lupper}{\overline{\La}}
\DeclareDocumentCommand{\lup}{O{x}}{\lupper(#1)}
\DeclareDocumentCommand{\llow}{O{p} O{\bpi}}{\lunder(#1, #2)}
\DeclareDocumentCommand{\lag}{O{x} O{p} O{\bpi}}{\La(#1, #2, #3)}

 \DeclareDocumentCommand{\pikpre}{O{t} O{t}}{\sigma^{#1}_k U_k(\pik^{#2 -1}, \pik^{#2})}
 \DeclareDocumentCommand{\pikfixt}{O{t} O{t}}{\sigma^{#1}_k U_k(\pik^{#2}, \pik)}
 \DeclareDocumentCommand{\pikfixtu}{O{t} O{t}}{(\sigma^{#1}_k + \mu_{\pik}) U_k(\pik^{#2}, \pik)}
 \DeclareDocumentCommand{\pikfixtt}{O{t} O{t-1}}{\pikfixt[#1][#2]}

\DeclareDocumentCommand{\xpre}{O{t} O{t}}{\eta^{#1} V(x^{#2 - 1}, x^{#2 })}
\DeclareDocumentCommand{\xfixt}{O{t} O{t}}{\eta^{#1} V(x^{#2}, x)}
\DeclareDocumentCommand{\xfixtu}{O{t} O{t}}{(\eta^{#1} + \mu_x) V(x^{#2}, x)}
\DeclareDocumentCommand{\xfixtt}{O{t} O{t-1}}{\xfixt[#1][#2]}

\DeclareDocumentCommand{\ppre}{O{t} O{t}}{\tau^{#1} W(p^{#2 - 1}, p^{#2})}
\DeclareDocumentCommand{\pfixt}{O{t} O{t}}{\tau^{#1} W(p^{#2}, p)}
\DeclareDocumentCommand{\pfixtu}{O{t} O{t}}{(\tau^{#1} + \mu_p) W(p^{#2}, p)}
\DeclareDocumentCommand{\pfixtt}{O{t} O{t-1}}{\pfixt[#1][#2]}

\def \cp {C_p}

\def \bT {\pmb{T}}
\def \Tk {T_k}
\def \bW {\pmb{W}}

\def \bO {\bar{\Omega}}

\def \qnex {\hat{q}}
\def \Hnex {\hat{H}}
\def \Hcur {\bar{H}}
\def \tilf {\widetilde{f}}
\def \tilcp {\widetilde{\cp}}
\def \tilmup {\widetilde{\mu_p}}
\def \ftiltk {\widetilde{f}_{t, k}}
\def \tiltau {\widetilde{\tau}}
\def \tilF {\tilde{F}}
\newcommand{\xbart}{\bar{x}_N}
\newcommand{\pbart}{\bar{p}_N}
\newcommand{\pibart}{\bar{\bpi}_N}

\def \Otp  {\widetilde{\Omega}_p}
\def \bB {(1 + \rtwo \cp \Opb)}
\def \bq {\bar{q}}
\def \bH {\bar{H}}
\def \bpi {\pmb{\pi}}
\def \bPi {\pmb{\Pi}}
\def \Bpi {\bPi}
\def \bU {\pmb{U}}
\def \bgstar {\pmb{g}^*}

\def \xcur {x_t} 
\def \xnext {x_{t+1}}
\def \xpre {x_{t-1}}
\def \znext {z_{t+1}}
\def \pcur {p_{t}}
\def \pnext {p_{t+1}}
\def \picur {\bpi_t}
\def \pinext {\bpi_{t+1}}
\def \upi {u_{\bpi}}
\def \up {u_p}
\def \ux {u_x}
\def \Hnext {H_{t+1}}
\def \Hcur {H_{t}}
\def \uH {u_H}

\def \bfe {e}

\def \OH {\Omega_H}
\def \bR {R}

\def \Ox {\Omega_X}
\def \Op {\Omega_P}
\def \Opi {\Omega_{\bPi}}

\def \Oyb {\bar{\Omega}_Y}
\newcommand{\smoothG}{\gmupik}

\def \smoothF {F_{\mupi, \mu_p}}

\def \Fu {F_{\mu}}

\def \pik {\pi_k}
\def \bpik {\bar{\pi}_k}

\def \hpi {\hat{\bpi}}
\def \Pik {\Pi(k)}

\def \mupi {\mu_{\bpi}}
\def \mup {\mu_p}

\def \Mt {M_{\pmb{T}}}
\def \Mpi {M_{\bPi}}
\def \lb {{\rm{lb}}}
\def \ub {{\rm{ub}}}
\newcommand{\innerk}[2]{\langle {#1,#2} \rangle_{\Tk}}
\newcommand{\innerT}[2]{\langle {#1,#2} \rangle_{\bT}}
\newcommand{\innerF}[2]{\langle {#1,#2} \rangle_{F}}

\def \pik {\pi_k}
\def \Pik {\Pi(k)}
\def \bg {\mathbf{g}}

\newcommand{\fu}{f_{\mu}}

\newcommand{\vup}{\bar{v}}
\newcommand{\vlo}{\underline{v}}
\newcommand{\xl}{x^l}
\newcommand{\xmd}{x^{md}}
\newcommand{\xu}{x^u}
\newcommand{\xb}{\bar{x}}

\newcommand{\Opib}{\bar{\Omega}_{\bpi}}
\newcommand{\Opb}{\bar{\Omega}_{p}}
\newcommand{\Mpib}{\bar{M}_{\bPi}}

\newcommand{\fmu}{f_{\mu}}
\newcommand{\Fmu}{F_{\mu}}
\newcommand{\rtwo}{\sqrt{2}}
\newcommand{\fP}{\mathcal{P}}
\newcommand{\blam}{\bar{\lam}}

\NewDocumentCommand\modify{m}{{#1}}
\usepackage{soul}
\NewDocumentCommand\delete{m}{}

\usepackage{pifont}

\NewDocumentCommand\mytextquotesingle{}{'}

\title{Efficient Algorithms for Distributionally Robust Stochastic Optimization with Discrete Scenario Support
\thanks{
This work was funded by Army Research Office W911NF-18-1-0223.}}

\author{
    Zhe Zhang\thanks{H. Milton Stewart School of Industrial \& Systems Engineering, 
                        Georgia Institute of Technology, Atlanta, GA, 30332 .
                        (email: {\tt jimmy\_zhang@gatech.edu}).}
    \and \framebox{Shabbir Ahmed}
    \and Guanghui Lan\thanks{H. Milton Stewart School of Industrial \& Systems Engineering, 
                             Georgia Institute of Technology, Atlanta, GA, 30332 .
                            (email: {\tt george.lan@isye.gatech.edu}).}
    }

\begin{document}
\maketitle
\begin{abstract}
Recently, there has been a growing interest in distributionally robust optimization (DRO) as a principled approach to data-driven 
decision making. In this paper, we consider a distributionally robust
two-stage stochastic optimization problem with discrete scenario support. 
 While much research effort has been devoted to tractable reformulations for DRO problems, especially those with continuous scenario support, 
 few efficient numerical algorithms are developed, and most of them can neither handle the non-smooth second-stage cost function nor the large number of scenarios $K$ effectively. 
 We fill the gap by reformulating the DRO problem as a trilinear min-max-max saddle point problem and
 developing novel algorithms that can achieve an
 $\bigO(1/\epsilon)$ iteration complexity which only mildly depends on $K$. 
 The major computations involved in each iteration of these algorithms can be conducted in parallel if necessary. 
 Besides,  for solving an important class of DRO problems with the Kantorovich ball ambiguity set,
 we propose a slight modification of our algorithms to avoid the expensive computation of the probability vector projection at the price of 
 an $\bigO(\sqrt{K})$ times more iterations. 
 Finally, preliminary numerical experiments are conducted to
 demonstrate the empirical advantages of the proposed algorithms.

 \vspace{.1in}

\end{abstract}

    \vspace{.07in}


\def\cX{{\cal X}}

\section{Introduction}

Two-stage stochastic programming (SP) problems are the most widely used stochastic optimization models in practice \cite{Shapiro2009Lecture}. In this paper, we consider a distributionally robust two-stage stochastic convex optimization problem with a finite set of scenarios $\{\xi_i\}^{K}_{i=1}$,
\begin{equation}\label{eq:xi_discrete}
\min_{x \in X} \left\{ f(x) \coloneqq f_0(x)  + \max_{p \in P} \tsum_{k=1}^{K} p_k g(x, \xi_k) - \phi^*(p)\right\},
\end{equation}
where $X \subset \R^n$ is a convex and compact feasible region for the first-stage decision variable $x$,
and $P \subset \R^{K}$ is a convex and compact ambiguity set for the scenario probability vector $p \in \R^{K}$.  
We assume that the first-stage cost function $f_0(\cdot)$ and the second-stage cost functions $g(\cdot, \xi_k)$ are proper closed convex  (\ccp) and Lipschitz continuous,
and that $\phi^*$\footnote{Notice that $\phi^*$ is usually identically zero in DRO problems, however we include it to handle some non-coherent risk measures for risk-averse stochastic programming problem.} is a simple \ccp\ function of $p$. The goal is to minimize the expected cost with respect to the worst probability vector in $P$. \\

Such a problem arises naturally  under the following situations.
\begin{itemize}

    \item \textbf{Data driven SP with finite scenario support.} We want to minimize the expected cost with respect to the true distribution $p^*$. However, $p^*$ is usually unknown, and only partial information about it  can be obtained from either historical observations or simulation. In this case, one can construct an $1-\alpha$ confidence ambiguity set $P_\alpha$, i.e., $p^* \in P_\alpha$ with a probability of at least $1-\alpha$, and solve for the DRO problem associated with $P_\alpha$. The true cost for DRO solution $\hat{x}$ would be less than the DRO cost with a probability of at least $1-\alpha$. There exist an expansive literature on such confidence ambiguity sets, including the Phi-divergence ball \cite{Pardo2018Statistical,Pflug2011Approximations}, the $\zeta$-distance ball \cite{Zhao2015Data}, and the hypothesis testing set \cite{Bertsimas2018Hypothesis}.

    \item \textbf{Data driven SP with continuous scenario support.} An important metric-based ambiguity set is the Kantorovich ball. This is because when $g(x, \xi)$ is Lipschitz continuous in $\xi$ for all $x$, the expected cost $\E_{p}[g(x, \xi)]$ is  Lipschitz continuous in $p$ with respect to the Kantorovich distance. In two-stage stochastic programming, the radius $\delta$ for the Kantorovich $P_\alpha$ ball \cite{Zhao2015W,ChenXu2018Decomposition}, the sufficient conditions for the Lipschitz-continuity of $g(x, \xi)$  and the convergence of DRO solutions to true solutions \cite{AliosXu2017Stability} are well studied. However, computing the DRO solution remains challenging because it involves finding the maximal in the infinite dimensional space of distributions. 
    One approach to address such a difficulty is to use a duality argument to simplify the problem to
    $$\min_{x \in X, \lam \geq 0} f_0(x) + \lam \delta +  \frac{1}{N} \tsum_{i=1}^{N} \max_{\xi\in \Xi} g(x, \xi) - \lam d(\xi_i, \xi),$$ 
    where $\lam$ is the Lagrange multiplier for the total transportation cost constraint and $d$ is the distance function. To solve the simplified problem, \cite{EsfahaniKuhn2017Data} reformulates it to a large deterministic convex problem, \cite{GaoKleywegt2016DRO} suggests using the mirror-prox algorithm and \cite{Zhao2015W} suggests using the Benders decomposition algorithm.  The successes of \cite{GaoKleywegt2016DRO,EsfahaniKuhn2017Data} hinge on the concavity of $g(x, \xi)$  with respect to $\xi$, while  \cite{Zhao2015W} carries some other structural assumptions on $g(x, \xi)$ and $\Xi$. These requirements can be restrictive. For example, the concavity of $g(x,\xi)$ is not satisfied even for a two-stage linear stochastic program with right-hand side uncertainty. A more general approach is to use a discrete grid of scenarios $\Xi_K$ to approximate the whole scenario space $\Xi$ and solve the DRO problem restricted to $\Xi_K$ \cite{ChenXu2018Decomposition,Zhao2015Data}. The approximation error  can be bounded by the Hausdorff distance between $\Xi_K$ and $\Xi$, so a fine grid, i.e., a large number of scenarios, is necessary for a moderately accurate solution.
    
        \item \textbf{Risk-averse SP with finite scenario support:} In finance, the preference for less risk can be formulated using a risk measure $\phi$, so the goal is to find a decision $x$ with minimal $\phi$. For example, in portfolio selection \cite{Markowitz1952Portfolio}, given a finite number of scenarios about possible returns $\{g(x, \xi_k)\}$, we want to select a portfolio $x$ with minimum  $\phi(g(x, \xi_1), ..., g(x, \xi_K))$. \modify{If such a risk measure is {\ccp} and monotone, say the piecewise linear dis-utility function, then we can use bi-conjugation \cite{Beck2017First} to rewrite the problem as $\min_{x \in X} \max_{p \in \R_+^K} \tsum_{k=1}^K p_k g_k(x) - \phi^*(p)$.} In addition if $\phi$ is a coherent risk measure \cite{Shapiro2009Lecture}, for example the average value-at-risk (AVaR), then $\phi^* \equiv 0$ and  $p${\mytextquotesingle}s domain must be a subset of the probability simplex. 
\end{itemize}

Now returning to \eqref{eq:xi_discrete}, we can simply denote the  $g(x, \xi_k)$  by $g_k(x)$.
In many cases, the function $g(x, \xi_k)$ may involve a linear transformation $T_k$ on $x$, for example,
the technology matrix in stochastic programming. Then it is often desirable to process
such a linear transformation differently from other nonlinear components of $g(x, \xi_k)$
in the design of algorithms. Therefore, we rewrite $g(x, \xi_k)$ as $g_k(\Tk x)$ to arrive at
the following equivalent reformulation of  \eqref{eq:xi_discrete},
\begin{equation}\label{eq:or_prob}
\min_{x \in X} \left\{ f(x)\coloneqq f_0(x) + \max_{p \in P} \sumk p_k  g_k(\Tk x) - \phi^*(p) \right\}.
\end{equation}
Apparently, if one does not need to process $\Tk$ 
separately or such a linear transformation does not exist, we can simply set $\Tk = I$
in \eqref{eq:or_prob}.

\begin{table}[]
\small

\caption{Theoretical Performance Comparison for Major Deterministic Algorithms}
\begin{threeparttable}
\modify{
    \begin{tabular}{l|l|l}
        \toprule
        Algorithm & Iteration Complexity & Most Expensive Computation in Each Iteration\tnote{1} \\
        \midrule
        Benders Decomposition \cite{Zhao2015Data,Lan2019SDDP} & $\bigO(1/\ep^n)$   & $K$ separable $m \times n$ LPs in parallel\\
         \midrule 
         Bundle Level \cite{Nem1995Bundle} & $\bigO(1/\ep^2) $ & $K$ separable $m \times n$ LPs in parallel \\
         Mirror Descent \cite{BenNem2001Lec} & $\bigO(1/\ep^2) $ & $K$ separable $m \times n$ LPs in parallel \\
         \midrule 
         Constraint PDHG \cite{LiuYuan2017Primal} & $\bigO(1/\ep)$ & One large-scale and non-separable $(Km) \times n$ QP \\
         Separabale PDHG \cite{ChenXu2018Decomposition} & $\bigO(K/\ep)$ & $K$ separable $m \times n$ QPs in parallel\\
         \midrule
         Euclidean SD \& SSL & $\bigO(\sqrt{K}/\ep)$ & $K$ separable $m \times n$ QPs  in parallel\\
         Entropy SD \& SSL & $\bigO(\sqrt{\log K}/\ep)$ & $K$ separable $m \times n$ QPs  in parallel \\
         \bottomrule
    \end{tabular}
    }
    \modify{
    \begin{tablenotes}
    \item[1] Based on solving distributionally robust two-stage LP.
    \item[2] The complexity of Benders decomposition (or Kelley's cutting plane method) was established in \cite{Lan2019SDDP} with $n$ being the dimension of the problem.
    \end{tablenotes}
    }
    \end{threeparttable}
    \label{tab:my_label}
\end{table}

Problem~(\ref{eq:or_prob}) is a convex-concave saddle point problem and can be solved by the mirror descent method \cite{BenNem2001Lec} or the bundle level method \cite{Nem1995Bundle} directly. However $g_k$ is often non-smooth, for example, the minimum objective of a linear program. So direct applications of these methods would lead to
an $\bigO(1/\epsilon^2)$ iteration complexity bound, which is independent
of the number of scenarios $K$. In each iteration, the   function values and sub-gradients for $\{g_k\}$ can be computed in parallel.

To improve the iteration complexity bound, Liu et al.  put the second-stage cost functions in the constraint  to obtain a composite bilinear saddle point problem in \cite{LiuYuan2017Primal},
\begin{equation}\label{eq:Liu_reformulation}
\min_{x \in X, v_k \geq g_k(\Tk x)}\max_{p \in P} f_0(x) + \tsum_{k=1}^{K} p_k v_k - \phi^*(p).
\end{equation}
They applied the primal-dual hybrid gradient (PDHG) algorithm in  \cite{ChambollePock2016Ergodic}  to obtain an $\bigO(1 / \ep)$ iteration complexity bound. However, this 
algorithm may not be practical because each iteration involves projecting $(x, v)$ onto a jointly constrained set,  $\{v_k \geq g_k(\Tk x), \forall k \}$. More recently, Chen et al. \cite{ChenXu2018Decomposition} address the non-separability issue by introducing a copy of $x$ for each scenario, $\{x_k\}$,  and uses Lagrange multipliers $\{\lam_k\}$ to enforce their consensus to arrive at the following reformulation:
\begin{equation}\label{eq:Hu_reformulation}
\min_{x_0,x_k \in X, v_k \geq g_k(\Tk x_k)} \max_{p \in P} \max_{\lambda_k \in R^n} \inner{v}{p} + f_0(x_0) + \sumk \inner{x_0 - x_k}{\lambda_k} - \phi^*(p).
\end{equation}
 The objective is jointly concave (linear) with respect to $(p, \lambda)$, so \eqref{eq:Hu_reformulation} is again a bilinear saddle point problem to which the PDHG algorithm can be applied. Moreover, the $(x_k, v_k)$ projections can be performed in parallel if needed. However such an approach still has two major limitations. Firstly, the combined $(p, \lambda)$ dual block prevents us from exploiting the special geometry of $P$, a subset of the probability simplex, to improve the iteration complexity bound{\mytextquotesingle}s dependence on $K$. More specifically, since the Euclidean Bregman distance is used in \cite{ChenXu2018Decomposition}, the \modify{radii} of both the primal feasibility region for $\{x_0, (x_1, v_1) ... (x_K, v_K)\}$ and the dual feasibility region for $\{(p_1, p_2 ... p_K); \lambda_1; \lambda_2; ... \lambda_K\}$ are $\bigO(\sqrt{K})$. So it follows from \cite{ChambollePock2016Ergodic} that the iteration complexity bound is $\bigO(K/\ep)$. Secondly, the projection onto a non-smooth function constrained set $\{v_k \geq g_k(\Tk x_k)\}$ in each iteration could be computationally expensive.

An interesting research problem is whether there exists an $\bigO(1/\ep)$ algorithm which can handle both the large number of scenarios and the non-smooth second-stage cost $g_k(\Tk x)$ effectively. Towards this end, we use bi-conjugation \cite{Beck2017First} to reformulate the non-smooth $g_k(\Tk x)$ as $\max_{\pik \in \Pik} \inner{\pik}{\Tk x}- g_k^*(\pik)$ to arrive at a trilinear saddle point problem,
 \begin{equation}\label{prob}
 \min_{x \in X}  f_0(x) + \underbrace{ \max_{p \in P} \tsum_{k=1}^{K} \max\limits_{\pik \in \Pik} \pk (\inner{\Tk x}{\pik} - g_k^*(\pik)) -\phi^*(p)}_{F(x)},
 \end{equation}
 where $\Pik$ is the domain of the conjugate function $g_k^*(\pik)$, \modify{$\Pi(k) := \{\pik\ | \ g_k^*(\pik) < \infty\}$}\footnote{\modify{Notice that if $g_k$ is a {\ccp} function, then $g_k^*$ must also be {\ccp}, so $\Pi(k)$ is closed and convex. Moreover, if $g_k$ is Lipschitz continuous, then $\Pi(k)$ must be bounded, i.e., $\Pi(k)$ is compact.}}. As compared to (\ref{eq:Hu_reformulation}), (\ref{prob}) is no longer jointly concave in $p$ and $\{\pi_k\}$, and the projection in $(p, \{\pik\})$ is not parallelizable. So the simple reduction to a convex-concave saddle-point problem is not possible. 
However, because $p$ is non-negative, the non-concave maximization in (\ref{prob}) can be evaluated efficiently in a sequential manner: given a $x \in X$, first maximize $\{\pik\}$ in parallel and then maximize $p$. 

In this paper, we take advantage of such a sequential structure by treating $p$ and $\{\pik\}$ as separate dual blocks and develop two new algorithms: a simple sequential dual (SD) method and a more complicated but more efficient sequential smoothing level (SSL) method.  The SD method extends the popular primal-dual method; it has a novel momentum step and an additional $p$-projection step. 
The SSL algorithm extends Nesterov's smoothing scheme to build a two-layer smooth approximation of (\ref{prob}) and then applies the accelerated prox-level method in \cite{Lan15Bundle} to an adaptively smoothed approximation of $f$. The SSL algorithm is   parameter-free. It is worth noting that bundle-level type methods are classical methods for solving two-stage stochastic programming problems, but they have not been studied for solving distributionally robust problems before. 

In addition, since $P$ is now a standalone block, we have more flexibility to exploit its favorable geometry to obtain either a better iteration complexity or cheaper computations in each iteration. More specifically, if $P$ is simple, we can use entropy $p$ projection to reduce
the iteration complexity bound to $\bigO(\sqrt{\log K}/\ep)$. If $P$ is the computationally challenging Kantorovich ball, we can substitute the expensive $p$ projection with 
a cheaper joint probability matrix projection at the price of increasing the iteration complexity to $\bigO(\sqrt{K}/\ep)$. Due to the separation of the $p$-block from the other blocks, only stepsize modifications are needed for our SD and SSL methods. To the best of our knowledge, all these complexity results appear to be new for solving trilinear saddle point problems given in the form of \eqref{prob}.

The paper is organized as follows. Section 2 proposes the simple sequential dual (SD) algorithm, and Section 3 develops the parameter-free sequential smoothing level (SSL) method. 
Section 4 introduces the specialized modifications of the SD and SSL algorithms for the challenging Kantorovich ball. Finally, encouraging numerical results are presented in Section 5 and concluding remarks are made in section 6.

\subsection{Notations and Assumptions} \label{sec_notation}

\modify{Throughout the paper, we use $\xstar$ denote an arbitrary optimal solution to \eqref{eq:or_prob}. For any convex function $f$ defined on $\cX$, we use $\partial f(x)$ to denote the set of all sub-gradients and use $f'(x)$ to denote an arbitrary element in $\partial f(x)$. If the set $\cX$ is associated with some norm $\norm{\cdot}_\cX$, we use $\norm{\cdot}_{\cX^*}$ to denote its dual norm. Moreover, we call $f$ $L$-smooth if it satisfies $f(x_1)- f(x_2)-\inner{x_1 - x_2}{ f'(x_2)} \leq \frac{L}{2}\norm{x_1 - x_2}_\cX^2$  for all $x_1, x_2 \in \cX$, and we call $f$ $\mu$-strongly convex if it satisfies $f(x_1)- f(x_2)-\inner{x_1 - x_2}{ f'(x_2)} \geq \frac{\mu}{2}\norm{x_1 - x_2}_\cX^2$  for all $x_1, x_2 \in \cX$.}

To take advantage of the geometry of $P$, we need the Bregman distance function. Given a closed and convex set Y\footnote{In general the Bregman distance function can be defined over any set, not necessary a closed and convex set. For the general definition, please refer to \cite{LanBook}. }, let $F: Y \rightarrow \R$ be differentiable and convex, and $1$-strongly convex over $\text{dom}(\partial F):=\{ y \in Y: \partial F(y) \neq \emptyset\}$ with respect to some $\norm{\cdot}_F$, the Bregman distance function  $d_F: \text{dom}(\partial F) \times Y \rightarrow \R$ is defined as
$$d_F(y_1, y_2) = F(y_2) - F(y_1) - \inner{F'(y_1)}{y_1 - y_2}.$$

In the following analysis, we will consider a general Bregman distance function $W(\cdot, \cdot)$ for $P$. Distance functions of practical interests 
consist of the Euclidean $W(p_1, p_2) := \norm{p_1 - p_2}^2$ and the entropy $W(p_1, p_2) := \tsum_{i=1}^{K} p_{1, i} \log (p_{2, i}/ p_{1, i})$, 
which are 1-strongly convex with respect to $\norm{\cdot}_2$ and $\norm{\cdot}_1$ respectively. For $X$ and $\Pik$, we will use  the Euclidean distance functions $V(x_1, x_2) := \norm{x_1 - x_2}_2^2/2$ and $U(\pi_{1,k}, \pi_{2,k}) := \norm{\pi_{1,k} - \pi_{2,k}}_2^2/2$ for simplicity. 

To facilitate analyzing how our algorithms scale with $K$, we need some scenario independent radii and operator norms. Define $\Ox^2 \coloneqq \max_{x \in X} V(x_0, x)$ and $\Op^2 \coloneqq \max_{p \in P} W(p_0, p)$ for some initial points $x_0$ and $p_0$. $\Ox$ is independent of $K$, but $\Op$ can depend on $K$. More specifically, if $p_0$ is the empirical distribution and $P$ is the whole probability simplex, then $\Op$ is $\bigO(1)$ for Euclidean $W$ and $\bigO(\sqrt{\log K})$ for entropy $W$.

For the multi-block $\{\pi_k\}$, we use boldface letters to denote the concatenation of individual scenarios:
$\bpi \coloneqq [\pi_1, \pi_2, \dots, \pi_K]$, $\bT \coloneqq  [T_1; T_2;\dots; T_K]$, $\bgstar(\bpi) \coloneqq [g^*_1(\pi_1), \dots, g^*_K(\pi_K)]$, and \modify{$\bPi:= \Pi(1) \times \Pi(2) \dots \times \Pi(K)$}. We  use the following shorthand notations for multi-scenario functions: $p \bT \pi \coloneqq \tsum_{k=1}^{K} p_k \Tk \pi_k$, $\innerT{x}{\bpi} \coloneqq [\inner{T_1 x}{\pi_1}, \inner{T_2 x}{\pi_2}, ..., \inner{T_K x}{\pi_K}] $ and $\bU(\bpi_1, \bpi_2) \coloneqq [U(\pi_{1, 1}, \pi_{2, 1}), U(\pi_{1, 2}, \pi_{2, 2}), ..., U(\pi_{1, K}, \pi_{2, K})]$ and their $k$-th components: $\innerk{x}{\bpi} \coloneqq \inner{\Tk x}{\pi_k}$ and $U_k(\bpi_1, \bpi_2) \coloneqq U(\pi_{1, k}, \pi_{2, k})$.  Let the multi-block (2,q)-norm be $\norm{\bpi}_{2, q} \coloneqq \norm{[\norm{\pi_1}_2, \norm{\pi_2}_2,  ..., \norm{\pi_K}_2]}_q$ , then the scenario independent radius and operator norm for $\bpi$ and $\bT$ are defined as : 
\begin{equation}\label{eq:pi_constant}
\text{
$\Opi^2 \coloneqq \max_{k \in K} \max_{\bpi \in \bPi} U_k(\bpi_0, \bpi)$ for some initial $\bpi_0 \in \bPi$, $\Mt\coloneqq \max_{k \in [K]} \norm{\Tk}_{2,2}$ and $\Mpi \coloneqq\max_{\bpi \in \bPi} \norm{\bpi}_{2, \infty}$.
}
\end{equation}
Because $g_k(\cdot)$ is {\ccp} and Lipschitz-continuous, every $\Pik$ is a convex closed and bounded, so  $\Mpi < \infty$.






\section{Sequential Dual Algorithm}
\modify{
In this section, we consider \eqref{prob} from a saddle point perspective:
\begin{equation}\label{eq:sad_La}
    \min_{x \in X} \max_{(p, \bpi) \in  P \times \Bpi} \{\La(x, p, \bpi) := f_0(x) + \tsum_{k=1}^{K} \pk (\inner{\Tk x}{\pik} - g_k^*(\pik)) -\phi^*(p)\}.
\end{equation}
One challenge is the non-concavity of $\La$ with respect to $(p, \bpi)$, so existing saddle point algorithms cannot be directly applied. However, upon a closer inspection, we find the ingredients required to design a primal-dual saddle point algorithm \cite{LanBook} still applicable due to the non-negativity of the $p$-block. More specifically, in Subsection 2.1, we show a duality relationship between $\La$ in \eqref{eq:sad_La} and $f$ in \eqref{eq:or_prob}, and a conversion from the primal-dual gap (see Definition \ref{def:gap}) to the functional optimality gap. Then in Subsection 2.2, we present a decomposition of  the primal-dual gap into individual optimality gaps of the $x,\ p$ and $\bpi$ blocks. These individual optimality gaps are composite linear, i.e., of the form $\inner{\cdot}{y} + h(\cdot)$, where $h(\cdot)$ are some simple convex functions. So, as is standard in first-order methods \cite{LanBook}, these quantities can be gradually decreased by iterative proximal updates. We introduce some novel momentum terms in these proximal updates, which then leads to the SD method. 
}
\subsection{Duality and Primal-Dual Function}
The following duality relationship between $f$ and $\La$ is straightforward because it boils down to switching the order of a non-negative weighted summation and a maximization. 
\begin{proposition}\label{pr:duality}
Let $f$ and $\La$ be defined in \eqref{eq:or_prob} and \eqref{eq:sad_La}, then the following statements hold for all $ x \in X$.
\begin{itemize}
\item[a)] Weak Duality: $f(x) \geq \La(x, p, \bpi)$ for all $p \in P, \bpi \in \Bpi$.
\item[b)] Strong Duality: $f(x) = \La(x, \pbar, \pibar)$ for some $\pbar \in P, \pibar \in \Bpi$.
\end{itemize}
\end{proposition}
\begin{proof}
We consider the strong duality first. Pick $\pikbar \in \partial g_k(\Tk x)$ such that $\inner{\Tk x}{\pikbar} - g_k^*(\pikbar) = g_k(\Tk x)$ and $\pbar \in \argmax_{p \in P} \sumk \pk g_k(\Tk x) - \phi^*(p)$, then it is easy to verify $\La(x, \pbar, \pibar) = f(x)$. \\
Next, we show the weak duality. Notice that
\begin{align*}
    f(x) =& f_0(x) + \max_{p \in P} \tsum_{k=1}^{K}\pk [ \max\limits_{\pik \in \Pik} \inner{\Tk x}{\pik} - g_k^*(\pik)] -\phi^*(p) \\
    \stackrel{(a)}{=}& f_0(x) + \max_{p \in P} \max\limits_{\pik \in \Pik} \tsum_{k=1}^{K}  \pk [\inner{\Tk x}{\pik} - g_k^*(\pik)] -\phi^*(p) = \max_{(p, \bpi) \in P \times \Bpi} \La(x, p, \bpi),
\end{align*}
where (a) follows from the non-negativity of $P$. So $f(x) = \max_{(p, \bpi) \in P \times \Bpi} \La(x, p, \bpi) \geq \La(x, p, \bpi)$ for any feasible $p$ and $\bpi$.
\end{proof}

We define a primal-dual gap function \cite{LanBook} for analyzing the saddle point problem \eqref{eq:sad_La} as follows.
\begin{definition}\label{def:gap}
Let $z:= (x, p, \bpi) \in Z := X \times P \times \bPi$ and $u := (u_x, u_p, u_{\bpi}) \in Z$. Then the primal-dual gap function is given by
$$Q(z, u) := \La(x, u_p, u_{\bpi}) - \La(u_x, p, \bpi).$$
\end{definition}
$Q(z, u)$ measures the saddle point optimality of $z$ in comparison to some $u$; if $z$ is a saddle point, then $Q(z, u) \leq 0$ for all feasible $u$. In our analysis, we use $Q(z, u)$ as an upper bound for the functional optimality gap, $f(x) - f(\xstar)$. With a carefully chosen $u$, we can show  $\sumt Q(\zt, u)$ providing an  upper bound for the optimality gap of an ergodic average solution $\xbart$, which is illustrated in the following proposition.
\begin{proposition}\label{pr:Q_func}
Let $u := (\xstar, \up, \upi)$ and a feasible sequence $\{\zt:=(\xt,\pt,\picur)\}$ be given. If  $\max_{(\up, \upi) \in P \times \Bpi} \sumt Q(\zt, u)\leq B$ for some finite $B$. Then the ergodic solution $\xbart := \sumt \xt / N$ satisfies
$$f(\xbart) - f(\xstar) \leq \tfrac{B}{N}.$$
\end{proposition}
\begin{proof}
From the strong duality result in Proposition \ref{pr:duality}, we have $\La(\xbart, \pbart, \pibart) = f(\xbart)$ for some $\pbart$ and $\pibart$. But $\La(x, \pbart, \pibart)$ is convex with respect to $x$, so it follows from the Jensen's inequality that $N f(\xbart) = N \La(\xbart, \pbart, \pibart) \leq \sumt \La(\xt, \pbart, \pibart)$. 
Moreover, the weak duality in Proposition \ref{pr:duality} implies that $f(\xstar) \geq \La(\xstar, \pt, \picur)\ \forall t$, so $N f(\xstar) \geq \sumt \La(\xstar, \pt, \picur)$. Therefore we have 
$$N (f(\xbart) - f(\xstar)) \leq \sumt \La(\xt, \pbart, \pibart) - \La(\xstar, \pt, \picur) = \sumt Q(\zt, (\xstar, \pbart, \pibart)) \leq B.$$
Dividing both sides by N, we get the desired result.
\end{proof}

\subsection{The Sequential Dual Method}\label{subsec:SD_conver}

The development of the sequential dual method (see Algorithm~\ref{alg: sd}) is inspired by
the following decomposition of $Q(z; u)$:
$Q(z; u) \equiv Q_x(z; u) + Q_p(z; u) + Q_\pi(z; u)$ where
\begin{equation}\label{eq:three_part}
\begin{split}
&Q_\pi(z; u) := \La(x, \up, \upi) - \La(x, \up, \bpi) = \inner{u_p}{\innerT{x}{u_{\bpi}} - \bgstar(u_{\bpi})} \underbrace{- \inner{u_p}{\innerT{x}{\bpi} - \bgstar(\bpi)}}_{\bpi \mbox{ gap }}.\\
&Q_p(z; u) := \La(x, \up, \bpi) - \La(x, p, \bpi) = \inner{u_p}{\innerT{x}{\bpi} - \bgstar(\bpi)} - \phi^*(u_p) \underbrace{-\inner{p}{\innerT{x}{\bpi} - \bgstar(\bpi)} + \phi^*(p)}_{p \mbox{ gap }}.\\
&Q_x(z; u) := \La(x, p, \bpi) - \La(\ux, p, \bpi) = \underbrace{f_0(x) + \inner{x}{p \bT \bpi}}_{x  \mbox{ gap }} - (f_0(u_x) + \inner{u_x}{p \bT \bpi}).  
\end{split}
\end{equation} 
Observe that inside each $Q_{(\cdot)}$ function, the under-braced terms associated with the $(\cdot)$ argument are of the form $\inner{y}{\cdot} + h(\cdot)$ for some simple convex $h(\cdot)$,  for example, $- \inner{u_p}{\innerT{x}{\cdot} - \bgstar(\cdot)}$ in the $\bpi$ gap. So we can use proximal updates to decrease them iteratively. However, at least two of $(x, p, \bpi)$ appear together in every decomposed gap term. Thus we need to use some \textit{guesses} for the other blocks if they have not been evaluated in the sequential update scheme, and care must be taken in designing those \textit{guesses} to ensure the cancellation of the consequent prediction errors. More specifically, 
given a sequence $\{z_i  \equiv (x_i, p_i, \bpi_i)\}_{i=0}^t$,
we propose the following sequential proximal update for the $\bpi,\ p$ and $x$ blocks (in that order) to obtain a possibly smaller $Q_x(z_{t+1}; u)$, $Q_p(z_{t+1}; u)$ and $Q_\pi(z_{t+1}; u)$.
\begin{enumerate}
    \item \textbf{$\bpi$ block:} we need to decrease the value of $- \inner{u_p}{\innerT{x_{t+1}}{\bpi} - \bgstar(\bpi)}$ in (\ref{eq:three_part}). But since $\up$ is non-negative, we might as well reduce every component of the vector $-\innerT{x_{t+1}}{\bpi} + \bgstar(\bpi)$ separately. Moreover, $x_{t+1}$ is currently unknown, so we use the \textit{guess} $\xcur + (\xcur - \xpre)$ to arrive at the following $\bpi-$proximal update step, i.e., Line 4 in Algorithm \ref{alg: sd}, 
    $$\pikp = \argmin_{\pik \in \Pik} - \innerk{2\xcur- \xpre}{\bpi}  + g^*_k(\bpi) + \sigma U_k(\picur, \bpi). $$
    \item \textbf{$p$ block:}  we wish to decrease the value of $-\inner{p}{\innerT{x_{t+1}}{\bpi_{t+1}} - \bgstar(\bpi_{t+1})} + \phi^*(p)$ in (\ref{eq:three_part}). Again, the information about $\innerT{x_{t+1}}{\bpi_{t+1}}$ is unavailable, so we use the \textit{guess} $\innerT{\xcur}{\picur} + (\innerT{\xcur}{\pinext} - \innerT{\xpre}{\picur})$ to obtain the following $p-$proximal update step, i.e.,
    Line 5 in Algorithm \ref{alg: sd},
    $$\pnext = \argmin_{p \in P} - \inner{p}{\innerT{\xcur}{\pinext} + \innerT{\xcur - \xpre}{\picur} - \bgstar(\pinext)} + \phi^*(p) + \tau W(\pcur, p). $$
    \item \textbf{$x$ block:} This is the simplest. We intend to decrease the value of $f_0(x) + \inner{x}{p_{t+1}\bT\bpi_{t+1}}$. But since we already know $(p_{t+1}, \bpi_{t+1})$ from the previous two updates, the $x$-proximal update step, Line 6 in Algorithm \ref{alg: sd}, is simply
    $$\xnext = \argmin_{x \in X} \inner{x}{\pnext \bT \pinext} + f_0(x) + \eta V(\xcur, x).$$
\end{enumerate}
The algorithm is named sequential dual method because both the $\bpi$ and $p$ blocks can be viewed as 
dual blocks and they need to be updated sequentially before the primal $x$ block can be updated.

\begin{algorithm}[htbp]
\caption{Sequential Dual Algorithm}
\label{alg: sd}
\begin{algorithmic}[1]
\Require $(x_0, p_0, \bpi_0) \in X \times P \times \bPi$ and stepsizes $\sigma, \tau, \eta >0$
\Ensure $(\bar{x}^N, \bar{p}^N, \bar{\bpi}^N)$
\State \textbf{Initialization} set $x_{-1} = x_0$.
\For{$t = 1,2,3 ... N$}
\State set $\txt = 2*\xtt - x_{t-2}$.
\State set $\pikt = \argmaxpik \dscenp[\txt] - \sigma U_k(\piktt, \pik), \quad\forall k \in [K].$
\State set $\ftiltk = \innerT{\xtt}{\picur} +\innerT{ (\xtt - x_{t-2})}{\bpi_{t-1}} - g^*_k(\pikt).$
\State set $\pt = \argmaxp \inner{p}{\ftiltk} $\modify{$-\phi^*(p)$}$ - \tau W(\ptt, p).$
\State set $\xt = \argminx \fir + \dseconx[x][p_{t}][\bpi_t]  + \eta V(\xtt, x).$
\EndFor
\State Return  $\bar{x}^N = \tsum_{t=1}^{N} \tfrac{\xt}{N}$.
\end{algorithmic}
\end{algorithm}




Our goal in the remaining part of this section is to analyze the convergence properties of the SD method. To highlight the dependence of the iteration complexity bound on $K$, we need to relate the dual norm $\norm{\bpi}_{2, W^*}$, which possibly depends on $K$, to $\norm{\bpi}_{2, \infty}$, which is independent of $K$.

\begin{definition}\label{def:cp}
Let $\norm{\cdot}_W$ be the norm associated with $P$, we call any $\cp \geq 0$  a norm adjustment constant for the ambiguity set P if it satisfies $\cp \norm{\bpi}_{2, \infty} \geq \norm{\bpi}_{2, W^*}$ for all $\bpi \in \bPi$.\\  
\end{definition}
In the following analysis, we use some specific choices of norm adjustment constants to make explicit dependence of the iteration complexity bound on $K$.
\begin{itemize}
\item[a)] When $\norm{\cdot}_1$ and entropy $W$ are used for $P$, we fix $\cp = $1. 
\item[b)] When $\norm{\cdot}_2$ and Euclidean $W$ are used for $P$, we fix $\cp = \sqrt{K}$.
\end{itemize}

\vgap

Proposition \ref{pr:sd} below shows that the SD method achieves an $\bigO(1/N)$ reduction in $Q(z_N; u)$. 
\begin{proposition}\label{pr:sd}
If the non-negative stepsizes satisfy
\begin{equation}\label{eq:stp_req}
\eta \geq \tfrac{\cp^2 \Mt^2 \Mpi^2}{\tau} + \tfrac{\Mt^2}{\sigma},
\end{equation}
where $\Mt$, $\Mpi$ and $\cp$ are defined in Section~\ref{sec_notation} and Definition \ref{def:cp},
then the following inequality holds for all $u \in Z$,
\begin{equation}\label{eq:Q_convergence}
\tsum_{t=1}^{N} Q(z_t; u) \leq \sigma \inner{u_p}{U_k(\bpi_0, u_{\bpi})} + \tau W( p_0, u_p) + \eta V(x_0, u_x).
\end{equation}
\end{proposition}
\begin{proof}
First, consider the three projection steps of Algorithm \ref{alg: sd} for a fixed iteration $t\geq 1$. 
In the $\bpi$ update step, it follows from the standard three point inequality of proximal update, e.g., Lemma 3.4 in \cite{LanBook}, that for a fixed $k$ scenario,
\begin{align*}
-\innerk{2\xcur - \xpre}{\pinext} + & g^*_k(\pinext) + \sigma (U_k(\picur, \pinext) + U_k(\pinext, \upi))\\
&\leq -\innerk{2\xcur - \xpre}{u_{\bpi}} + g^*_k(u_{\bpi}) + \sigma U_k(\picur, \upi),
\end{align*}
or equivalently,
\begin{align*}
&\innerk{\xnext}{\upi - \pinext} -  g_k^*(\upi) + g_k^*(\pinext) \\
& \ \ \ \ \leq \sigma(U_k(\picur, \upi) - U_k(\pinext, \upi)) + 
(\sigma U_k(\picur, \pinext) + \innerk{\xnext - (2\xcur - \xpre)}{\upi - \pinext}).
\end{align*}
Summing up both sides with weight $\up$, we get
\begin{align*}
 Q_\pi(&\znext; u) \leq \sigma \inner{\bU(\picur, \upi) - \bU(\pinext, \upi)}{\up} \\
 &+ \innerT{\innerT{\xnext - \xcur}{\upi-\pinext} - \innerT{\xcur -\xpre}{\upi - \picur}}{\up} + \epsilon_\pi(\pinext), \stepcounter{equation}\tag{\theequation}\label{pi_offset} 
 \end{align*}
where
\begin{align*}
& \quad  \epsilon_\pi(\pinext) = \inner{\up}{\underbrace{\innerT{\xcur -\xpre}{\pinext - \picur} - \sigma \bU(\picur, \pinext)}_{\leq \tfrac{1}{2\sigma } \norm{\xcur - \xpre}^2_2 \Mt^2 \text{ for each component}}} \leq  \tfrac{1}{2\sigma} \norm{\xcur - \xpre}^2_2 \Mt^2.  \stepcounter{equation}\tag{\theequation}\label{ep_pi}
\end{align*}
Next in the $p$ update step, again it follows from Lemma 3.4 in \cite{LanBook} that
\begin{align*}
\inner{\up - \pnext}{\innerT{\xcur}{\pinext} + \innerT{\xcur - \xpre}{\picur} - \bgstar(\pinext)} + &\phi^*(\pnext) - \phi^*(\up) + \tau (W(\pcur, \pnext) + W(\pnext, \up)) \\
& \leq \tau W(\pcur, \up).
\end{align*}
After adding  $\inner{\up - \pnext}{\innerT{\xnext}{\pinext}}$ to both sides of the inequality, we have
\begin{align*}
Q_p(\znext; u) &\leq \inner{\up - \pnext}{  \innerT{\xnext - \xcur}{\pinext} - \innerT{\xcur - \xpre}{\picur}} + \tau (W(\pcur, \up) - W(\pnext, \up) - W(\pcur, \pnext)) \\
& \leq \tau (W(\pcur, \up) - W(\pnext, \up)) + (\inner{\up - \pnext}{ \innerT{\xnext - \xcur}{\pinext}} - \inner{\up - \pcur}{\innerT{\xcur - \xpre}{\picur}}) \\
& \quad + \epsilon_p(\pnext), \stepcounter{equation}\tag{\theequation}\label{p_offset}
\end{align*}
where
 \begin{align*}
 \epsilon_p(\pnext)&= \inner{\pnext - \pcur}{\innerT{\xcur - \xpre}{\picur}}  - \tau W(\pcur, \pnext) \\
 &\leq \norm{\pnext - \pcur}_W  \norm{\xcur - \xpre}_2 \norm{[\norm{T_1 \bpi_{t,1}}_2, ... ,\norm{T_K \bpi_{t,K}}_2]}_{W^*} - \tau W(\pcur, \pnext)\\
 &\leq \tfrac{1}{2 \tau} \norm{\xcur - \xpre}^2_2 (Cp \Mt \Mpi)^2. \stepcounter{equation}\tag{\theequation}\label{ep_p}
\end{align*}
Moreover, when computing $\xnext$ in $x$ update step, we can obtain the following simple inequality
\begin{align*}
Q_x(\znext; u) 
&= \inner{\xnext - \ux}{\pnext \bT \pinext} + f_0(\xnext) - f_0(\ux) \\
&\leq \eta (V(\xcur, \ux) - V(\xnext, \ux)) - \eta V(\xcur, \xnext).
\stepcounter{equation}\tag{\theequation}\label{x_offset}
\end{align*}
Finally, summing up \eqref{pi_offset}, \eqref{p_offset}, \eqref{x_offset} for $t = 1, 2, 3, \ldots, N$ and applying the telescoping cancellation, we have
\begin{align*}
\tsum_{t=0}^{N-1}& Q(z_{t+1}; u) \leq \sigma \inner{\up}{\bU(\bpi_{0}, \upi)} + \tau W(p_0, \up) + \eta V(x_0, \ux) - \eta V(x_{N}, \ux)\\
& + \inner{\up}{\innerT{x_{N} - x_{N-1}}{\upi - \bpi_N} - \sigma \bU(\bpi_N, \upi)} + (\inner{\up - p_N}{\innerT{x_N - x_{N-1}}{\bpi_N}} - \tau W(p_N, \up)) \\
 &+ \underbrace{\tsum_{t=1}^{N-1} (\epsilon_p(p_{t+1}) + \epsilon_\pi(\bpi_{t+1}) - \eta V(x_{t-1}, x_{t}))}_{
 \text{Notice that $\epsilon_p(p_1) = 0$ and  $\epsilon_\pi(\bpi_1) = 0$ because $x_0 = x_{-1}$.}
 } - \eta V(x_{N-1}, x_{N}). \stepcounter{equation}\tag{\theequation}\label{pr3:sum_up}
\end{align*}
Observe that the stepsize requirement $\eta \geq \cp^2 \Mt^2 \Mpi^2/ \tau + \Mt^2/\sigma$ implies that the following parts of \eqref{pr3:sum_up} are smaller than $0$:
\begin{align*}
&\tsum_{t=2}^{N} \epsilon_p(p_t) + \epsilon_\pi(\bpi_t) - \eta V(x_{t-1}, x_{t}) \\
&\quad \leq \tsum_{t=2}^{N} \tfrac{1}{2\tau} \norm{x_{t} - x_{t-1}}^2_2 (Cp \Mt \Mpi)^2  + \tfrac{1}{2\sigma} \norm{x_t - x_{t-1}}^2_2 \Mt^2 - \tfrac{\eta}{2} \norm{x_t - x_{t-1}}^2_2 \leq 0. \\
&\inner{\up}{\innerT{x_{N} - x_{T-1}}{\upi - \bpi_N} - \sigma \bU(\bpi_N, \upi)} \\
&\quad + (\inner{\up - p_N}{\innerT{x_N- x_{T-1}}{\bpi_N}} - \tau W(p_N, \up)) - \eta V(x_{N-1}, x_{N})\leq 0. 
\end{align*}
So (\ref{eq:Q_convergence}) follows by substituting the previous two inequalities and $-\eta V(x_{N}, \ux)\leq 0$ into \eqref{pr3:sum_up}.
\end{proof}


\modify{The next theorem suggests a stepsize choice for Algorithm \ref{alg: sd} and shows its convergence in terms of function value gap.}
{

    \renewcommand{\sumt}{\tsum_{t=1}^{N}}
    



    \begin{theorem}\label{thm:sd}
     \modify{
    If we set
    \begin{equation}\label{eq:step_size}
    \sigma = \Mt \tfrac{\Ox}{\Opi}, \tau = \Mt\Mpi \cp \tfrac{\Ox}{\Op}, \text{ and } \eta = \Mt \Mpi \cp \tfrac{\Op}{\Ox} + \Mt \tfrac{\Opi}{\Ox},
    \end{equation}
    then after $N$ iterations of Algorithm \ref{alg: sd}, we have
    \begin{equation}\label{eq:sd_thm_res}
    f(\xbart) - f(\xstar) \leq \tfrac{2 \Ox \Mt}{N} (\Opi + \cp \Mpi \Op).
    \end{equation}
}
    \end{theorem}

    \begin{proof}\modify{Observe that the stepsize choices in \eqref{eq:step_size} satisfies the requirement in \eqref{eq:stp_req} and $\Opi^2$, $\Op^2$, and $\Ox^2$ are upper bounds for $\inner{u_p}{\bU(\bpi_0, u_{\bpi})}, W( p_0, u_p)$ and $ V(x_0, \xstar)$ for any feasible $ u_p, u_{\bpi}$. So it follows from Proposition \ref{pr:sd} that 
    $$\sumt Q(\zt, (\xstar, \up, \upi)) \leq \sigma \Opi^2 + \tau \Op^2 + \eta \Ox^2 \quad \forall (\up, \upi) \in  P \times \bPi.$$
    Thus Proposition \ref{pr:Q_func} implies that $f(\xbart) - f(\xstar) \leq \tfrac{\sigma \Opi^2 + \tau \Op^2 + \eta \Ox^2}{N}$. The bound \eqref{eq:sd_thm_res} then follows from substituting the stepsize choices into the preceding inequality.}
    \end{proof}
}

\vgap

We remark here that, by using $\rtwo \Mpi$ as an upper bound for $\Opi$, the above convergence rate could be further simplified to $\Ox \Mt \Mpi(\rtwo + \cp \Op)/N$, i.e., $\bigO((1 + \Op \cp)/N)$ if we ignore constants independent of $K$. Then substituting in the values of $\cp$ and $\Op$, the iteration complexity bounds become $\bigO({\sqrt{\log K}}/{\ep})$ for entropy W and $\bigO({\sqrt{K}}/{\ep})$ for Euclidean W.
{\color{blue}
It is also worth noting that 
the aforementioned rate of convergence for SD seems to be tight for solving problem \eqref{eq:sad_La} since the $\bigO(1/N)$ rate of convergence is not improvable even for solving
the simpler convex-concave bilinear saddle point problems \cite{Nemirovsky1992,Ouyang2019Lower}.}

\section{Sequential Smooth Level Method}
In this section, we view \eqref{prob} from the perspective of a structured non-smooth problem,
\begin{equation}\label{eq:F}
    F(x) := \max_{p \in P} \tsum_{k=1}^{K} \max\limits_{\pik \in \Pik} \pk (\inner{\Tk x}{\pik} - g_k^*(\pik)) -\phi^*(p).
\end{equation}
\eqref{eq:F} contains an additional maximization layer  than those considered by Nesterov in \cite{Nestrov2004Smooth}. Moreover, these two maximization layers cannot be combined because of non-separability and non-concavity issues. So the current smoothing technique are not directly applicable. To address such a difficulty, Subsection 3.1 extends the Nesterov's framework to build a two-layer smoothing scheme for $F$ and analyzes its smooth approximation properties with respect to a sequence of points. Such a sequence-based approach helps us to determine a suitable smoothing scheme for the encountered points, rather than for the whole feasible region. 
 
Another challenge is deciding the smoothing parameters to balance the conflicting goals of a small approximation gap (for a sound solution) and a small Lipschitz smoothness constant (for fast convergence). In fact, to calculate an optimal choice of those parameters for a fixed smoothing scheme, we would need to know the distance to the output solution even before the algorithm is run, which is preposterous. Subsection 3.2 resolves such a difficulty by introducing a parameter-free bundle level type algorithm that operates on a dynamically smoothed $F$, where the smoothness parameters adjust in an on-line fashion to the encountered points.

\subsection{Sequential Smoothing Scheme}
 
 By a smooth approximation for a non-smooth function $f$, we mean a convex function $\tilde{f}$ which is both $L$-smooth and close to $f$ everywhere on its domain.

\begin{definition}
Let $f$ be a convex function on $\cX \subset R^n$ equipped with norm $\norm{\cdot}_\cX$. We call a convex function $\tilde{f}$ its \textit{$(\alpha, \beta)$-domain smooth approximation}  if
\begin{itemize}
    \item[a)] $\norm{\grad \tilde{f}(x_1) - \grad \tilde{f}(x_2)}_{\cX^*} \leq \alpha \norm{x_1 -x_2}_\cX \ \ \forall x_1, x_2 \in \cX$,
    \item[b)] $\tilde{f}(x) \leq f(x) \leq \tilde{f}(x) + \beta  \ \ \forall x \in \cX.$
\end{itemize}
\end{definition}

For our purpose of designing an adaptive smoothing algorithm, we need a weaker notion of smooth approximation. More specifically, since we use the accelerated proximal level (APL) method in \cite{Lan15Bundle} as the backbone of the SSL algorithm, it is useful to note that the $L$-smoothness constant is only used to bound the upper curvature constants associated with the linearization centers $\{x^l_t\}$ and the search points $\{\xmd_t\}$. So we should focus  on the upper curvature constant and the approximation gap associated with these points and define an $(\alpha, \beta)$-\textit{sequence smooth approximation}.

\begin{definition}
Let $f$ be a convex function on $\cX \in R^n$ equipped with norm $\norm{\cdot}_\cX$ and let $\{(x^l_t, \xmd_t)\}_{t=1}^{N}$ be some sequence of points in $\cX$. Then we call a convex differentiable function
$\tilde{f}$ an \textit{$(\alpha, \beta)$-\textit{sequence smooth approximation}} of $f$ over $\{(x^l_t, \xmd_t)\}_{t=1}^{N}$ if the following conditions hold.
\begin{itemize}
\item[a)] $\tilf(\xmd_t) - \tilf(x^l_t) - \inner{\grad \tilf(x^l_t)}{ \xmd_t  - x^l_t} \leq \frac{\alpha}{2} \norm{ \xmd_t  - x^l_t}^2_\cX$.
\item[b)] $\tilde{f}(\xmd_t) \leq f(\xmd_t) \leq \tilde{f}(\xmd_t) + \beta  \ \ \forall t \in [N]$.
\end{itemize}
\end{definition}

It is worth noting that if $\tilf$ is an $(\alpha, \beta)$-\textit{domain smooth approximation}, then it must be an $(\alpha, \beta)$-\textit{sequence smooth approximation} for all sequences. Moreover, if $\tilf$ is an $(\alpha, \beta)$-\textit{sequence smooth approximation} for all singleton sequences $\{x^l_t, \xmd_t\}_{t=1}^{1}$, then it must be an $(\alpha, \beta)$-\textit{domain smooth approximation}. Because of such a close relationship, we use the generic name ``smooth approximation'' when referring to both of them. \\

Now we develop the two-layer smooth approximation scheme for (\ref{prob}). Let us briefly review Nesterov's smoothing scheme in \cite{Nestrov2004Smooth} for the following structured non-smooth function $H:\cX \to \R$,
\begin{equation}
H(x) = \max_{y \in Y} \inner{x}{Ay} - \psi(y), 
\end{equation}
where $\psi(y)$ is some simple {\ccp} function defined on $Y$.
 Nesterov suggests adding a $\mu$-multiple of some 1-strongly convex term $\omega$  to the inner $y$-maximization
 to obtain
 \begin{equation}\label{eq:NesSmo}
H_\mu(x) = \max_{y \in Y} \inner{x}{Ay} - \psi(y) - \mu \omega(y).
\end{equation}
The following properties of $H_\mu$ are established in Theorem 1 of \cite{Nestrov2004Smooth}.
\begin{lemma}\label{lm:NesSmo}
Let $\omega$ be 1-strongly convex with respect to some $\norm{\cdot}_\omega$, then the following statements hold for $H_\mu$ defined in \eqref{eq:NesSmo}.
\begin{itemize}
    \item [a)] $H_{\mu}(\cdot)$ is convex and continuously differentiable with gradient $H'_{\mu}(x) = A^T \hat{y}$, where $\hat{y}$ is the unique solution to the maximization problem in $H_{\mu}(x)$.
    \item [b)] For any $x_1, x_2 \in \cX$ and their corresponding maximizers in $H_\mu(\cdot)$, $\hat{y}_1, \hat{y}_2$, we have \\
    $\inner{A(x_1 - x_2)}{\hat{y}_1 - \hat{y}_2} \geq \mu \inner{\grad \omega(\hat{y}_1) - \grad \omega(\hat{y}_2)}{ \hat{y}_1 - \hat{y}_2} \geq \tfrac{\mu}{2} \norm{\hat{y}_1 - \hat{y}_2}_\omega^2$.
    \item [c)] If $\Omega^2_Y := \max_{y \in Y} \omega(y)$, $H_{\mu}(\cdot)$ is an $(\tfrac{\norm{A}^2_{\omega, X}}{\mu}, \mu \Omega_y^2)$-\textit{domain smooth approximation} of $H(x)$. 
\end{itemize}
\end{lemma}

\vgap

Returning to our problem (\ref{prob}),  the subgradient of $F(x)$  is $p \bT \bpi$. So to make it Lipschitz continuous, 
we can consider the following product rule type decomposition\footnote{Recall that $p\bT\bpi$ is not matrix multiplication; it is merely a short hand for $\sum_{i=1}^{K} p_k T_k \pi_k$. However the decomposition in \eqref{eq:smo_decompostion} is valid because $\sum_{i=1}^{K} p_k T_k \pi_k$ is linear with respect to $p$ and $\bpi$.}:
\begin{equation}\label{eq:smo_decompostion}
p_1 \bT \bpi_1 - p_2 \bT \bpi_2 = \underbrace{(p_1 - p_2)}_{p\ smoothing} \bT \bpi_1  + p_2 \bT \underbrace{(\bpi_1 - \bpi_2)}_{\bpi\ smoothing}.
\end{equation}
If we smooth both the $p$-block and the $\bpi$-block, $p \bT \bpi$  should be a Lipschitz continuous function of $x$. More specifically, we consider the following $F_{\mupi, \mup}(x)$ smooth approximation,
 \begin{equation}\label{eq:smooth_F}
\begin{aligned}
&\smoothG(x) := \max_{\pik \in \Pik} \inner{\pik}{\Tk x} - g_k^*(\pik) - \mupi U(0, \pik), & \text{($\bpi$ smoothing)}  \\
&F_{\mupi, \mup}(x) := \max_{p \in P} \tsum_{k=1}^{K} p_k \smoothG(x) - \phi^*(p) - \mu_p W(\bar{p}, p), \text{ for some } \bar{p} \in P. &\text{($p$ smoothing)}
\end{aligned}
\end{equation}
Notice that proxy center for $U(\bpik, \pik)$ is set to $\bpik := 0$. 
Such a choice allows us to use $\Mpi/\rtwo$ to bound $\Opi$ so that we need to dynamically estimate only two problem parameters, $\Mpi$ and $\Op$.

Now we analyze the properties of the proposed smooth approximation (\ref{eq:smooth_F}). The following domain smooth approximation properties of (\ref{eq:smooth_F}) are direct consequences of Lemma \ref{lm:NesSmo}. 
\begin{lemma}\label{lm:separate_smo} The following statements hold for $F_{\mupi, \mup}$ in \eqref{eq:smooth_F}.
\begin{itemize}
    \item[a)] As a function of x, $\gmupik$ is a $(\Mt^2/\mupi, \mupi \Mpi^2/2)$-domain smooth approximation of $g_k(\Tk x)$.
    \item[b)] As a function of $\gmupi(x)$, $F_{\mup, \mupi}(\cdot)$ is a $(\norm{I}_{g,W}/\mup, \mup \Op^2)$-domain smooth approximation of $F(\gmupi(x)) := \max_{p \in P} \tsum_{k=1}^{K} p_k \smoothG(x) - \phi^*(p)$.\\
\end{itemize}
\end{lemma}
\begin{proof}
Part b) is clear. 

For part a), Lemma \ref{lm:NesSmo} implies that $\smoothG$ is a $(\norm{\Tk}^2_{2,2}/\mupi, \mupi (\max_{\pik \in \Pik} U(0, \pik))$-\textit{domain smooth approximation} of $g_k$. But $\Mt$ and $\Mpi^2/2$ are upper bounds for $\norm{\Tk}_{2,2}$ and $\max_{\pik \in \Pik} U(0, \pik)$ for all $k$, so a) follows immediately.
\end{proof}

\vgap

Just like the chain rule in calculus, we need the following technical result to reduce the above $p$-block $L$-smoothness property with respect to $\gmupi(x)$ 
to that with respect to $x$. 

\begin{lemma}\label{lm:g_continuity} Let $\norm{\cdot}_{W^*}$ be the dual norm of the $p$-block. Then for any feasible ($x_1$, $x_2$) and their corresponding maximizers in the definition of $\gmupi$, ($\bpi_1$, $\bpi_2$), we have 
\begin{equation}\label{eq:g_continous}
\norm{\gmupi(x_1) - \gmupi(x_2)}_{W^*} \leq \cp \Mt \max  \{\norm{\bpi_1}_{2, \infty},\norm{\bpi_2}_{2, \infty}\} \norm{x_1 - x_2}_2. 
\end{equation}
\end{lemma}
\begin{proof}
First, we derive the following Lipschitz-continuity constant for each $\smoothG$: 
$$|\smoothG(x_1) - \smoothG(x_2)| \leq \max\{\norm{\Tk^\intercal \pi_{1, k}}_2, \norm{\Tk^\intercal \pi_{2, k}}_2\} \norm{x_1 - x_2}_2.$$
Because $|\smoothG(x_1) - \smoothG(x_2)|$ is the difference of two maximal values attained over the same domain, we can use the maximizer of the larger value in place of the maximizer of the smaller value to derive an upper bound. More specifically, if $\smoothG(x_1) \geq \smoothG(x_2)$, 
then 
\begin{align*}
\smoothG(x_1) &- \smoothG(x_2) \\
&:= \inner{\pi_{1, k}}{\Tk x_1} - g_k^*(\pi_{1, k}) - \mupi U(0, \pi_{1, k}) -  \max_{\pik \in \Pik}( \inner{\pik}{\Tk x_2} - g_k^*(\pik) - \mupi U(0, \pik))\\
&\leq \inner{\Tk \pi_{1, k}}{x_1 - x_2} 
\leq \norm{x_1 - x_2}_2 \norm{\Tk \pi_{1, k}}_2.
\end{align*}
 A similar bound can also be obtained when $\smoothG(x_1) \leq \smoothG(x_2)$. So we have
\begin{align*}
|\smoothG(x_1) - \smoothG(x_2)| \leq \max\{\norm{\Tk \pi_{1, k}}_2, \norm{\Tk \pi_{2, k}}_2\} \norm{x_1 - x_2}_2 \leq \Mt \max  \{\norm{\bpi_1}_{2, \infty},\norm{\bpi_2}_{2, \infty}\} \norm{x_1 - x_2}_2.
\end{align*} Finally (\ref{eq:g_continous}) follows from the definition of $\cp$ in Definition \ref{def:cp}.
\end{proof}

\vgap

Combining the previous two results, we obtain the following \textit{sequence smooth approximation} property of
 (\ref{eq:smooth_F}).

\begin{proposition}\label{pr:smooth_cst}
Let $\{x^l_t, \xmd_t\}_{t=1}^N$ be given. Let $\{\hp^u_t, \hpi^u_t\}$ be the maximizers for $\{F(\xmd_t)\}$ in (\ref{eq:F}), and let $\{p^l_t, \pi^l_t\}$ and $\{p^u_t, \pi^u_t\}$ be the maximizers for $\{\smoothF(x^l_t)\}$ and $\{\smoothF(\xmd_t)\}$ in \eqref{eq:smooth_F}. If $\Opb^2 \geq \max_{t \in [N]} W(\bar{p}_t, p)$ and $\Mpib \ge \max_{t \in [N]} \max \{\norm{\hpi^u_t}_{2,\infty}, \norm{\bpi^u_t}_{2,\infty},\norm{\bpi^l_t}_{2,\infty}\}$, then $\smoothF$ is a $({2\Mt^2}/{\mupi} + {2\cp^2 \Mpib^2 \Mt^2}/{\mup}, \mup \Opb^2 + \mupi {\Mpib^2}/{2})$-\textit{sequence smooth approximation} of F on $\{x^l_t, \xmd_t\}_{t=1}^N$.
\end{proposition}
\newcommand{\tilFmu}{\tilF_{\mup, \mupi}}

\begin{proof}
Let a $t \in [N]$ be given. For simplicity, we use $x_1$ and $x_2$ to denote $\xmd_t$ and $x^l_t$ and use $(p_1, \bpi_1)$ and $(p_2, \bpi_2)$ to denote their corresponding maximizers in $\smoothF$ \eqref{eq:smooth_F}.
Denoting $\tilFmu(x, p, \bpi) := \tsum_{k=1}^{K} p_k (\innerT{\bpi}{ x} - g_k^*(\bpi) - \mupi V(\bpik, \bpi)) -\phi^*(p) - \mu_p W(\bar{p}, p)$, we have the following decomposition for the upper curvature error,
\begin{align*}
&\smoothF(x_1) - \smoothF(x_2)  - \inner{\grad \smoothF(x_2)}{x_1 - x_2} \\
&= \tilFmu(x_1, p_1, \bpi_1) - \max_{p, \bpi} \tilFmu(x_2, p, \bpi) - \inner{p_2\bT \bpi_2}{x_1 - x_2}\\
&\stackrel{(a)}{\leq} \tilFmu(x_1, p_1, \bpi_1) - \tilFmu(x_2, p_1, \bpi_1) - \inner{p_2\bT \bpi_2}{x_1 - x_2} \\
& = \inner{p_1\bT \bpi_1}{x_1 - x_2} - \inner{p_2\bT \bpi_2}{x_1 - x_2} \\
& = \inner{p_1\bT \bpi_1 - p_2\bT \bpi_2}{x_1 - x_2}\\
& = \underbrace{\inner{p_1 \bT (\bpi_1 - \bpi_2)}{x_1 - x_2}}_{A} + \underbrace{\inner{ (p_1 - p_2) \bT \bpi_2}{x_1 - x_2}}_{B},
\end{align*} 
where (a) follows from using $\tilFmu(x_2, p_1, \bpi_1)$ as a lower bound for $\max_{p, \bpi} \tilFmu(x_2, p, \bpi)$. 
To bound $A$, we conclude from Lemma \ref{lm:separate_smo}.a) that
\begin{equation}\label{tA}
 A \leq \norm{x_1 - x_2}_2 \tsum_{k=1}^{K} p_{k, 1} \max_{k \in K} \norm{ \Tk' (\pi_{1, k} - \pi_{2, k})}_2 \leq \tfrac{\Mt^2}{\mupi} \norm{x_1 - x_2}_2^2.
 \end{equation}
To bound $B$, we use Lemma \ref{lm:separate_smo}.b) and Lemma \ref{lm:g_continuity} to obtain
\begin{align*}
\begin{split}
\norm{p_1 - p_2}_W & \leq \tfrac{\norm{I}_{W^*, W}}{\mup}\norm{\gmupi(x_1) - \gmupi(x_2)}_{W^*}\\
 &\leq \tfrac{1}{\mup} \cp \Mt \max  \{\norm{\bpi_1}_{2, \infty},\norm{\bpi_2}_{2, \infty}\} \norm{x_1 - x_2}_2 \leq \tfrac{1}{\mup} \cp \Mt \Mpib \norm{x_1 - x_2}_2,
 \end{split}
\end{align*}
which implies that 
\begin{align}\label{tB}
\begin{split}
B &= \inner{p_1 - p_2}{\innerT{\bpi_{2}}{x_1 - x_2}}\\
 &\leq \norm{p_1 - p_2}_W \norm{[\norm{T_1^\intercal \pi_{2, 1}}_2 \norm{x_1 - x_2}_2,..., \norm{T_K^\intercal \pi_{2, K}}_2 \norm{x_1 - x_2}_2]}_{W^*}\\
 &\leq \tfrac{1}{\mup} \cp \Mt \Mpib \norm{[\norm{T_1^\intercal \pi_{2, 1}}_2,..., \norm{T_K^\intercal \pi_{2, K}}_2]}_{W^*} \norm{x_1 - x_2}_2^2  \\
 & \stackrel{(b)}{\leq} \tfrac{1}{\mup} (\cp \Mt \Mpib)^2  \norm{x_1 - x_2}_2^2 ,
 \end{split}
\end{align}
where (b) follows from the the definition of $\cp$ in Definition \ref{def:cp}.
Combining \eqref{tA} and \eqref{tB}, we obtain the desired upper-curvature constant of ${2\Mt^2}/{\mupi} + {2\cp^2 \Mpib^2 \Mt^2}/{\mup}$. 
Moreover, it is easy to see that for a given $\xmd_t$, we have $$\gmupik(\xmd_t) \leq g_k(\xmd_t) \leq \gmupik(\xmd_t) + \mupi U(0, \hpi^u_{t,k}) \leq \gmupik(\xmd_t) + \mupi \tfrac{\Mpib^2}{2},$$ 
and hence
\begin{align*}
\smoothF(\xmd_t) \leq F(\xmd_t) &\leq \smoothF(\xmd_t) + \tsum_{k=1}^{K} \hp^u_{t, k} \mupi \tfrac{\Mpib^2}{2} + \mup W(\bar{p}, \hp^u_t)\\ &\leq \smoothF(\xmd_t) + \mup \Opb^2 + \mupi \tfrac{\Mpib^2}{2}.
\end{align*} 
\end{proof}\\
Using $\Mpi$ and $\Op$ as upper bounds for $\Mpib$ and $\Opb$ for any $(x_1, x_2) \in X \times X$, we obtain the following \textit{domain smooth approximation} properties of (\ref{eq:smooth_F}) below as an immediate corollary.

\begin{corollary}
$\smoothF$ is a $(2\Mt^2/\mupi + 2\cp^2 \Mpi^2 \Mt^2/\mup, \mup \Op^2 + \mupi \Mpi^2/2)$-domain smooth approximation of F.
\end{corollary}
\vgap

The need to select two smoothing parameters, $\mup$ and $ \mupi$, makes (\ref{eq:smooth_F}) rather complicated. 
The next result shows a reduction to a single-parameter smoothing scheme by fixing an optimal ratio between $\mup$ and $\mupi$. 

\begin{lemma}\label{lm:opt_ratio}
Let $\smoothF$ be a $(2\Mt^2/\mupi + 2\cp^2 \Mpib^2 \Mt^2/\mup, \mup \Opb^2 + \mupi \Mpib^2/2)$-smooth approximation of F, then the optimal ratio 
is \[\tfrac{\mup}{\mupi} = \tfrac{\cp \Mpib^2}{\sqrt{2} \Opb}\].
\end{lemma}
\begin{proof}
To achieve the smallest gap while maintaining a Lipschitz constant at $1/\mu$, we solve the following optimization problem analytically by the KKT condition,
\begin{align*}
\min_{\mup, \mupi \geq 0} &\left\{\mup \Opb^2 + \mupi \tfrac{\Mpi^2}{2}: \tfrac{2\cp^2 \Mt^2 \Mpib^2}{\mup} + \tfrac{2\Mt^2}{\mupi} = \tfrac{1}{\mu} \right\}.
\end{align*}
\end{proof}

Using the above optimal ratio, $F_\mu$ defined below is then a $(\Mt^2/\mu, (1 + \rtwo \cp \Opb)^2  \Mpib^2 \mu)$-\textit{sequence smooth 
approximation} of $F$:
\begin{equation}\label{eq:opt_smo}
\Fu(x) = F_{\bar{\mu}_p, \bar{\mu}_\pi}(x) \ \mbox{with} \  \bar{\mu}_\pi :=  \mu (2 + 2\rtwo \cp \Opb)  \ \mbox{and} 
\ \bar{\mu}_p := \mu (\rtwo + 2 \cp \Opb){\Mpib^2 \cp}  / { \Opb} .
\end{equation}
Moreover, if we replace $\Opib$ and $\Mpib$ with their uniform upper bounds, $\Op$ and $\Mpi$, then (\ref{eq:opt_smo}) must be a $(\Mt^2/\mu, (1 + \rtwo \cp \Op)^2  \Mpi^2 \mu)$-\textit{domain smooth approximation} of $F$. 
Observe that the smooth approximation properties of $F_\mu$ in \eqref{eq:opt_smo} and $H_\mu$ in (\ref{eq:NesSmo}) studied by Nesterov \cite{Nestrov2004Smooth} differ only by a constant factor, therefore any variant of Nesterov's accelerated
gradient method could be applied to a fixed $f_\mu := f_0 + F_\mu$ to achieve an $\bigO(\cp \Op/\ep)$ iteration complexity bound. However, this approach suffers from the same drawback as 
Nesterov's smoothing scheme in \cite{Nestrov2004Smooth}, i.e., one has to use conservative estimates of $\Opi$ and $\Op$ to guarantee an $\bigO(\ep/2)$ uniform approximation gap. 
This usually leads to a large $L$-smoothness constant for $F_\mu$, and thus a slow convergence. To address this shortcoming, we present 
in the next subsection a novel SSL algorithm which operates on an adaptively smoothed $\Fu$.

\subsection{Sequential Smooth Level Method}

\begin{algorithm}[htbp]
\caption{SSL Phase}
    \label{alg:ssl_phase}
    \begin{algorithmic}[1]

    \Require  $\xb, \lb, \Mpib^2, \bar{\Omega}^2_p, \bar{\lam}$
    \Ensure $\tilde{x}, \tilde{\lb}, \tilde{M}^2_\pi, \tilde{\Omega}^2_p, \tilde{\lam}$
    \State \textbf{Initialization:} set $\xu_0 := \xb,\ \vup_0 := f(\xu_0),\ \vlo_0 := \lb,\ l := \tfrac{1}{2} (\vlo_0 + \vup_0), \theta := \tfrac{1}{2}$,
    and $\mu := \tfrac{\theta (\vup_0 -l )}{\Mpib^2 (1 + \sqrt{2} \bar{\Omega}_p \cp)^2 \bar{\lam}}$. Set the initial localizer $X'_0 := X$ and  $t:= 1$. 
    \While{True}
    \parState{\textit{\textbf{Update the lower bound}}: set $\xl_t := (1 - \alpha_t) \xu_{t-1} + (\alpha_t) x_{t-1}$. Evaluate $f_\mu$ at $\xl_t$ to get $(p^l_t, \bpi^l_t)$ and construct a supporting function $s(\xl_t, x) := f_0(x) + F_{\mu}(\xl_t) + \inner{\grad \Fmu(\xl_t)}{x - \xl_t}$.
      Let $\textit{s}_t := \argmin_{x \in X'_{t-1}} s(\xl_t, x)$ and $\vlo_t := \max \{\vlo_{t-1}, \min\{\textit{s}_t, l\}\}$. \\
      If $\vlo_t \geq l - \theta (l - \vlo_0)$, \textbf{return} $(\xu_{t-1}, \vlo_t, \Mpib^2, \bar{\Omega}^2_p, \bar{\lam})$.
    }
    \State \textit{\textbf{Update the prox center}}: set $x_t := \argmin_{x \in X'_{t-1}, s(\xl_t, x) \leq l} \{V(x_0, x)\}$.
    \parState{\textit{\textbf{Update the upper bound}}: set $\xmd_t := (1 - \alpha_t) \xu_{t} + \alpha_t x_t$ and evaluate $f$ and $\fmu$ at $\xmd_t$ to get $(\hat{p}_t^{md}, \hat{\bpi}_{t}^{md})$ and $(p_t^{md}, \bpi_{t}^{md})$.
     Set $\vup_t := \min \{\vup_{t-1}, f(\xmd_t)\}$ and choose $\xu_t$ such that $f(\xu_t) = \vup_t$. \\
    If $\vup_t \leq l + \theta (\vup_0 - l)$, \textbf{return} $(\xu_{t}, \vlo_t,\Mpib^2, \bar{\Omega}^2_p, \bar{\lam})$.
    }
    \parState{\textit{\textbf{Check $\pi$ radius}}: let $\tilde{M}^2_{\bpi} := \max_{k \in K} \max\{ U(0, \pi^{l}_{k,t}), U(0, \hat{\pi}^{md}_{k,t}), U(0, \pi^{md}_{t, k})\}$. \\
    If $\tilde{M}^2_{\bpi} > {\Mpib^2}$, \textbf{return} $(\xu_{t}, \vlo_t, 2\tilde{M}^2_{\bpi}, \bar{\Omega}^2_p, \bar{\lam})$.
    }
    \parState{\textit{\textbf{Check $p$ radius}}: let $\tilde{\Omega}^2_p := W(\bar{p}, p^{md}_t)$.
    If $\tilde{\Omega}^2_p > \bar{\Omega}^2_p$, \textbf{return} $(\xu_{t}, \vlo_t, \Mpib^2, 2\tilde{\Omega}^2_p, \bar{\lam})$.
    }
    \parState{\textit{\textbf{Check aggressiveness param $\lam$}}: if $f_\mu(x^{md}_t) \leq l + \tfrac{\theta}{2}(\vup_0 - l)$, \textbf{return} $ (\xu_{t}, \vlo_t, \Mpib^2, \bar{\Omega}^2_p, 2\bar{\lam})$.}
    \parState{\textit{\textbf{Update the localizer}}: choose an arbitrary $X'_t$ such that $\underline{X}_t \subset X'_t \subset \bar{X}_t$ where \\
     $\underline{X}_t  := \{x \in X'_{t-1}: s(x^l_t, x)\leq l\} \quad \text{and} \quad \bar{X}_t :=  \{x \in X: \inner{\grad_{x = x_k} V(x_0, x)}{x - x_k} \geq 0\}$.}
    \State Set $t := t + 1$.
    \EndWhile
\end{algorithmic}
\end{algorithm}
The bundle level method maintains both an upper and a lower bound on $f_*$. The upper bound $\bar{f}$ is the minimum function value of all the encountered points, while the lower bound $\underline{f}$ is the minimum value of a lower approximation model $h(x)$, namely bundle, consisted of all evaluated cutting planes for $f$. In each iteration, $\bar{f}$ and $\underline{f}$ are used to construct a level set, say $\{x:h(x) \leq l:=(\bar{f} + \underline{f}) /2\}$, in which the next search point and the next cutting plane will be found. By repeating this process many times, the gap  between such lower and upper bounds can be decreased to $\ep$, upon which an $\ep$-optimal solution must have been found. 

To build an adaptive smoothing algorithm, we follow \cite{BenNem2005NERM,Lan15Bundle} to partition the  iterations into phases, inside which some important parameters are fixed. In \cite{BenNem2005NERM}, the constant $l$ for defining a level set is fixed to allow the use of a restricted memory localizer. A phase of the NERML algorithm in \cite{BenNem2005NERM} is terminated only when the upper bound or the lower bound has made enough progress  to warrant a new $l$ for the next phase. In our SSL algorithm, similar to \cite{Lan15Bundle}, we fix both $l$ and the smooth approximation function $\Fu$ in a phase. The smoothing parameters $(\bar{\mu}_p, \bar{\mu}_\pi)$ in \eqref{eq:opt_smo} are computed using current radii estimates. If these radii estimates are violated by a new point, we  also terminate the current phase such that a more appropriate smoothing scheme can be constructed for the next phase. So each phase has two goals: to reduce the gap between the lower and upper bounds, and to update the radii estimates and hence the smoothing scheme. 

\begin{itemize}
\item \underline{Radius Update}: Line 6, 7, and 8 of the SSL Phase in Algorithm \ref{alg:ssl_phase}. For each phase, we should construct a \textit{sequence smooth approximation} $\Fu$ with the smallest possible upper curvature constant for fast termination. In the USL method in \cite{Lan15Bundle}, the $L$-smoothness constant of the smooth approximation $H_\mu$ is $\bigO(\Oyb)$ and  the estimate of $\Oyb$ is updated only when it is absolutely necessary; the objective value achieved by the smooth approximation is well below the upper bound termination threshold, i.e., $H_\mu(\xu_t) \leq l + {\theta}(\bar{v}_0 - l)/2$, while the true objective value is above the upper bound termination threshold, i.e., $H(\xu_t) \geq l + {\theta}(\bar{v}_0 - l)$. In this way, \cite{Lan15Bundle} underestimates $\Oyb$ to encourage an aggressively small upper curvature constant. Our situation is different because we need both accurate estimates of radii $\Mpib$ and $\Opb$ to determine the optimal ratio between $\mup$ and $\mupi$ in Lemma \ref{lm:opt_ratio} and an aggressively small upper curvature constant for fast convergence. So we create a separate variable $\lam$ to control the aggressiveness of the smooth approximation and use $\Mpib$ and $\Opb$ for estimating $\Mpi$ and $\Op$ only. The radius update block in Algorithm \ref{alg:ssl_phase} thus has two components: 1) Line 6 and 7 check our estimates against the distances of encountered points to the fixed smoothing centers, $\pbar$ and $0$. 
Once we find any violations, the violated radius estimate is doubled and the phase is terminated so that the next phase can construct a more appropriate smooth approximation. 2) Line 8 updates the aggressiveness parameter $\lam$ in the same fashion as the $\Oyb$ update in the USL method. It is doubled only when the objective value achieved by the smoothed approximation is well below the upper bound termination threshold, $f_\mu(x^{u}_t) \leq l + {\theta}(\vup_0 - l)/2$, while the true objective value is above the upper bound termination threshold, $f(x^{u}_t) \geq l + \theta (\vup_0 - l)$, i.e., the approximation gap is too large.


\item \underline{Gap Reduction}: Line 3, 4, 5, and 9 of the SSL Phase in Algorithm \ref{alg:ssl_phase}. This is essentially the composite accelerated proximal level (APL) method \cite{Lan15Bundle} applied to the composite smooth approximation function $f_\mu := f_0 + \Fu$. Notice that, similar to Nesterov's accelerated gradient method \cite{Nes83}, we use three sequences of points $\{x^l_t\}$, $\{\xmd_l\}$ and $\{x_t\}$; we pick $\xl_t := (1 - \alpha_t) \xu_{t-1} + (\alpha_t) x_{t-1}$ to construct the composite cutting plane model and $\xmd_t := (1 - \alpha_t) \xu_{t} + \alpha_t x_t$ to evaluate the objective value. It is shown in \cite{Lan15Bundle} that the following convergence result holds for any composite smooth function, and our $f_\mu$ in particular.
\end{itemize}

\begin{lemma}\label{lm:APL}
Let $\alpha_t = 2/(t+1)$, and also let $\{\xl_t\}$, $\{\xmd_t\}$  and $\{\xu_t\}$ be the sequences of points generated by Algorithm \ref{alg:ssl_phase} before it terminates. If $\{\xl_t,\xmd_t\}^N_{t=1}$ satisfy
$\fu(\xmd_t) - s(\xl_t; \xmd_t) \leq \frac{M}{2} \norm{\xmd_t - \xl_t}^2$ for some $M \geq 0$, then we have
$$\fu(\xu_N) - l \leq  \tfrac{M \Ox^2}{N^2} .$$
\end{lemma}

Before Algorithm \ref{alg:ssl_phase} terminates, our estimates $\Mpib$ and $\Opb$ satisfy assumptions in Proposition \ref{pr:smooth_cst}, so $\Fu$ in (\ref{eq:opt_smo}) is a $({\Mt^2}/{\mu}, {(1 + \rtwo \cp \Opb)^2}  \Mpib^2 \mu)$-sequence smooth approximation of $F$. Therefore our choice of $\mu := \tfrac{\theta (\vup_0 -l )}{\Mpib^2 (1 + \sqrt{2} \bar{\Omega}_p \cp)^2 \bar{\lam}}$ in Algorithm \ref{alg:ssl_phase} implies that
\begin{equation}\label{eq:alg_fu_bd}
\text{$\Fu$ is a } (\tfrac{\Mt^2 \Mpib^2 (1 + \sqrt{2} \bar{\Omega}_p \cp)^2 \blam}{\theta (v_0 - l)}, \tfrac{\theta (v_0 -l)}{ \blam}) \text{-\textit{sequence smooth approximation} of $F$ over $\{\xl_t, \xmd_t\}.$}
\end{equation}
By substituting $M = \tfrac{\Mt^2 \Mpib^2 (1 + \sqrt{2} \bar{\Omega}_p \cp)^2 \blam}{\theta (v_0 - l)}$ into Lemma \ref{lm:APL}, we can obtain the following bound on the number of iterations performed by the SSL Phase in Algorithm \ref{alg:ssl_phase}.

\begin{proposition}\label{pr:ssl_phase_finite_ter}
Let $\alpha_t := {2}/{(t+1)}$ and $\Delta_0 := f(\xb) - \lb$. The SSL Phase in Algorithm \ref{alg:ssl_phase} terminates in at most ${( 4\rtwo  \Ox \Mt \Mpib \sqrt{\blam} \bB)}/{\Delta_0}$ iterations.
\end{proposition}

\begin{proof}
Assuming that  all other termination conditions have not been reached, then Algorithm \ref{alg:ssl_phase} will terminate in Line 8 if $\fu(x^{md}_N) - l \leq \tfrac{1}{2} \theta (\vup_0 - l) := \tfrac{1}{8} \Delta_0$.
So it follows from \eqref{eq:alg_fu_bd} and Lemma \ref{lm:APL} that the maximum number of iterations, $N_{SSL}$, is bounded by 
$$\tfrac{1}{8} \Delta_0 \leq \tfrac{\Mt^2 \Mpib^2 (1 + \sqrt{2} \bar{\Omega}_P \cp)^2 \blam \Ox^2}{ \tfrac{1}{4} \Delta_0} \tfrac{1}{N^2_{SSL}}.$$
After some simplification, we obtain the desired finite termination bound.
\end{proof}

\begin{algorithm}[htbp]
\caption{Sequential Smoothing Level Method}
    \label{alg:ssl}
    \begin{algorithmic}[1]
    \Require  $\xb_0 \in X$, tolerance $\epsilon >0$, initial estimate $\bO_{p, 0}^2 \in (0, \Op^2], \bO_{\pi, 0}^2 \in (0, \Opi^2], \bq_0 \in (0, 1]$ and  $\lam_0 \in (0, 1)$
    \Ensure $\xb$, an $\epsilon-$suboptimal solution 
    \State \textbf{Initialization} Set $\xb_1 = \argmin_{x \in X}\{ h(\xb_0, x) = f_0(x) + F(\xb_0) + \inner{\grad F(\xb_0)}{x - \xb_0}\}$, $\lb_1 = h(\xb_0, \xb_1)$ and $\ub_1 = \min \{f(\xb_0), f(\xb_1)\}$. Set $s = 0$.
    \While{True}
    \State If $\ub_s  - \lb_s \leq \epsilon$, \textbf{terminate} with $\xb = \xb_s$.
    \parState{Set $(\xb_{s+1}, \lb_{s+1}, \bar{M}^2_{\pi, s+1}, \bO^2_{p, s+1}, \blam_{s+1}) = \mbox{SSL-Phase}(\xb_{s}, \lb_s, \bar{M}^2_{\pi, s}, \bO^2_{p, s}, \blam_{s})$ and set $\ub_{s+1} = f(\xb_{s+1})$.}
    \State Set $s = s+ 1$.
    \EndWhile
\end{algorithmic}
\end{algorithm}

There are two ways for the SSL Phase Algorithm \ref{alg:ssl_phase} to terminate. If it terminates in Line 3 or Line 5, the gap between the lower and upper bounds is reduced by a factor of at least $1/4$. So we call it a gap reduction phase. Otherwise, if it terminates in Line 6, 7, or 8, then one of the estimates $\Opb^2$, $\Mpib^2$ and $\blam$ must have been enlarged by a factor of two. So we call it an estimate enlargement phase. Because $\Opb^2$ or $\Mpib^2$ is doubled only when a $p$ or a $\bpi$ exceeding its current radius estimate is found, $\Opb^2$ and $\Mpib^2$ are upper bounded by $2\Op^2$ and $2\Mpi^2$ respectively. Similarly, since the difference between $f$ and $f_\mu$ on observed points ${\xmd_t}$ is at most $\theta (v_0 - l)/(\blam)$ (by (\ref{eq:alg_fu_bd})), the termination condition, $f(\xmd_t) > l + \theta(v_0 -l)$ and $f_u(\xmd_t) < l + \tfrac{\theta}{2}(v_0 -l)$ in Line 8 can be satisfied only if $\blam < 2$, i.e., $\blam$ must be bounded by 4. Therefore, if we repeat the SSL Phase Algorithm  with updated \lb,  $\xb$, $\Mpib$, $\Opb$ and $\blam$ in Algorithm \ref{alg:ssl}, there will only be a finite number of estimate enlargement phases, and the gap reduction phases should reduce the gap to $\ep$ eventually. Thus we have the following iteration complexity result for the SSL Algorithm.

\begin{theorem}\label{thm:ssl}
Let $\alpha_t := 2/(t+1)$ and $f_0$ be Lipschitz continuous with constant $M_0$. To obtain an $\epsilon$-suboptimal solution, 
the SSL algorithm requires at most $\fP_s = \log_{4/3} (2 \Ox (\rtwo M_0 + \Mt \Opi (\rtwo + \cp \Op))/\epsilon)$ gap reduction phases and $\fP_N = \log_{2} (\Opi^2/\bO^2_{\pi, 0})
 + \log_{2} (\Op^2/\bO^2_{p, 0}) + \log_2 (1/\bq_0) + 4$ parameter enlargement phases. 
In total, the number of iterations performed by Algorithm \ref{alg:ssl_phase}  can be bounded by
$$
16  (7\rtwo +  \log_2 \tfrac{\Op^2}{\bO^2_{p, 0}}\rtwo + 2)  \tfrac{\Ox \Mt \Opi (1 + 2 \cp \Op)}{\ep}.
$$
\end{theorem}

\begin{proof}
Firstly, let us consider the gap reduction phases. A bound for the initial gap is 
\begin{align*}
\ub_1 - \lb_1  &\leq f(\xb_1) - f(\xb_0) - \inner{f'(\xb_0)}{\xb_1 - \xb_0} \leq \inner{f'(\xb_1) - f'(\xb_0)}{\xb_1 - \xb_0} \\
&= \underbrace{\inner{f_0'(\xb_1) - f_0'(\xb_0)}{\xb_1 - \xb_0}}_{A} + \underbrace{\inner{\grad F(\xb_1) - \grad F(\xb_0)}{\xb_1 - \xb_0}}_{B}.
\end{align*}
By the Cauchy Schwartz inequality and the triangle inequality, the following bounds on $A$ and $B$ hold,
\begin{align*}
A &\leq (\norm{f_0'(\xb_1)}_2 + \norm{f_0'(\xb_0)}_2) \norm{\xb_1 - \xb_0}_2 \leq 2 M_0 \rtwo \Ox = 2\rtwo M_0 \Ox,\\
B &\leq \norm{\xb_1 - \xb_0}_2 (\norm{(p_1 - p_2) \bT \bpi_1}_2 + \norm{ p_2 \bT (\bpi_1 - \bpi_2)})\\
&\leq \rtwo \Ox (\rtwo \Mt \Op \Mpi \cp + 2 \Mt \Mpi) = 2 \Mt \Mpi \Ox (\rtwo +  \cp \Op).
\end{align*}
So we have $\ub_1 - \lb_1 \leq  2 \Ox (\rtwo M_0 + \Mt \Opi  (\rtwo +  \cp \Op))$, and that number of gap reduction phases are bounded by
$\fP_s = \log_{4/3} (2 \Ox (\rtwo M_0 + \Mt \Opi  ( \rtwo + \cp \Op))/\epsilon) $. 
For the estimate enlargement phases, as discussed before, the upper bounds for $\Mpib^2, \ \Opb^2$ and $ \blam$ are $2\Opi^2,\  2\Op^2$ and $ 4$ respectively, hence there are at most
$\fP_N = \log_{2} (\Opi^2/\bO^2_{\pi, 0}) + \log_{2} (\Op^2/\bO^2_{p, 0}) + \log_2 (1/\blam_0) +4$ phases.
\vgap

Next, we develop separate bounds on the total number of iterations required for the gap reduction phases, $\Mpib^2$ enlargement phases, $\Opb^2$ enlargement phases and $\blam$ enlargement phases. 
For the gap reduction phases, let $g_1 \leq g_2 \leq g_3 \leq ... \leq g_S$ be their indices in Algorithm \ref{alg:ssl}. Then by the construction of Algorithm \ref{alg:ssl}, the initial gap for each phase $\Delta_s := f(x_s) -\lb_s$ must satisfy $\Delta_{g_i} \geq \ep (\tfrac{3}{4})^{i - S}$. Thus it follows from Proposition~\ref{pr:ssl_phase_finite_ter} and the relations $\bar{M}_{\pi, s} \leq \rtwo \Mpi$, $\bO_{p, s} \leq \rtwo \Op$ and $\blam_s \leq 4\ \forall s$ that the total number of iterations in the gap reduction phases is bounded by
\begin{align*}
& \tsum_{i=1}^{S} \tfrac{8\rtwo \Ox \Mt  (\rtwo \Mpi) (1 + 2 \cp \Op) \rtwo}{\tfrac{3}{4}^{i - S} \ep } \\
& \leq \tsum_{j=0}^{\infty} (\tfrac{3}{4})^j \tfrac{8\rtwo \Ox \Mt  (\rtwo \Mpi) (1 + 2 \cp \Op) \rtwo}{\ep} \leq 8\rtwo (8) \tfrac{ \Ox \Mt  ( \Mpi) (1 + 2 \cp \Op) }{\ep}.
\end{align*}
For the $\Mpib$ enlargement phases, let $s_1 \leq s_2 \leq ... \leq s_L$ be their indices in Algorithm \ref{alg:ssl}. Similar to the previous analysis, we use the geometric upper bound
$\bar{M}_{\pi, s_i} \leq \Mpi (1/\rtwo)^{L-i}$ and uniform upper bounds $4$, $\rtwo \Op$, $1/\ep$ for $\blam_s, \bar{\Omega}_{p, s}, 1/\Delta_{s}, \forall s$ to conclude that  
the number iterations in the $\Mpib$ enlargement phases is bounded by
\begin{align*}
& \tsum_{i=1}^{L} \tfrac{8\rtwo \Ox \Mt  ( \Mpi) (1 + 2 \cp \Op) \rtwo}{\ep} (\tfrac{1}{\rtwo})^{L-i} \\
& \leq \tsum_{j=0}^{\infty} (\tfrac{1}{\rtwo})^j \tfrac{8\rtwo \Ox \Mt  ( \Mpi) (1 + 2 \cp \Op) \rtwo}{\ep} \leq 16(\rtwo + 1)  \tfrac{ \Ox \Mt  ( \Mpi) (1 + 2 \cp \Op) }{\ep}.
\end{align*}
Similarly, the number of iterations in the $\blam$ enlargement phases can be bounded by $ 16\rtwo  ( \Ox \Mt  ( \Mpi) (1 + 2 \cp \Op) )/\ep$.
Next, since there are at most $ (\log_{2} (\Op^2/\bO^2_{p, 0}) + 1)$ $\Opb$-enlargement phases and the number of iterations in each phase is bounded 
uniformly by $16 \rtwo ( \Ox \Mt  (\rtwo \Mpi) (1 + 2 \cp \Op) )/\ep$, the number of iterations in the $\Opb$ enlargement phases should be bounded by
$$ 16 (\log_{2} \tfrac{\Op^2}{\bO^2_{p, 0}}\rtwo + \rtwo)  \tfrac{ \Ox \Mt  ( \Mpi) (1 + 2 \cp \Op) }{\ep}.$$
The desired iteration complexity bound follows by adding up these individual bounds.
\end{proof}

\vgap

We remark here that the above iteration complexity bound has the same dependence on $\ep$ and $K$ as that of the SD algorithm, i.e., $\bigO((1 + \cp \Op) / \ep)$, which does not seem to be improvable for solving general trilinear saddle point problems.

\section{Adaptation For Kantorovich Ball}
In the previous sections, we assumed $P$ being simple such that the $p$ proximal update is easy. However, this is not always the case; when $P$ is the Kantorovich ball, a projection onto  it is expensive. To avoid such an expensive computation, we propose to use the joint probability matrix projection instead. Because of  the standalone $P$ block in our reformulation \eqref{prob}, such an alternative update can be incorporated into the SD and SSL algorithms with only a change of stepsizes.

\subsection{Kantorovich Ball and Joint Probability Matrix Proximal Update}

Given $K$ scenarios and a distance matrix $D \in R_+^{K \times K}$, i.e., $D_{i,j} = d(\xi_i,\xi_j)$, the $\delta$-Kantorovich ball around the empirical distribution vector $\bar{p} = [\tfrac{1}{K}, \tfrac{1}{K} ... \tfrac{1}{K}]$ is
\begin{align}\label{eq:kan_formulation}
\begin{split}
P_\delta := \left\{p \in R^{K}_+\ s.t\ \right.  \exists&\ H \in\R_+^{K\times K},  \\
&\  \bar{p}_i = \tsum_{j=1}^{K} H_{i, j}, \ \forall i, \quad\quad\quad\quad\text{(source constraints)}\\
 &\ p_j = \tsum_{i=1}^{K} H_{i, j}, \ \forall j, \quad\quad\quad\quad \text{(target constraints)}\\
 &\left.\innerF{D}{H}\leq \delta \right\}, \quad\quad\quad\quad\quad\ \  \ \text{(transportation cost constraint)}
\end{split}
\end{align}
where $\innerF{D}{H}$ represents the Frobenius inner product, $\tsum_{i=1}^{K} \tsum_{j=1}^{K} D_{i,j} H_{i,j}$.
Since every row and every column of the joint probability matrix $H$ is constrained by a linear equality, 
the computation for the $p$ proximal update, $\argmax_{p \in P_\delta} \inner{c}{p} + W(\pbar, p)$\footnote{Notice that $\phi^*\equiv 0$ for Kantorovich ball ambiguity set.},  is not separable across scenarios.
In particular, when $W$ is the Euclidean distance function, we have to solve  a quadratic program (QP) with $\bigO(K^2)$ variables and $\bigO(K)$ linear constraints, 
and when $W$ is the entropy distance function, we have to solve an exponential cone problem of the same size. In fact, even checking whether a given $p$ is inside $P_\delta$ involves solving an expensive optimal transport problem. 
\vgap

Alternatively, we can remove the target constraints in \eqref{eq:kan_formulation} by representing $ p$ in terms of $H$ and consider a proximal update of $H$. Moreover, the rows of $H$, i.e.,
$\{H_i\}$, would become separable after we dualize the single transportation cost constraint. 

More specifically, to implement a separable $H$ proximal update, we need a row separable Bregman distance function $\bW$ for $H$ 
constructed from the Bregman distance function $W$ for $p$, $$\bW(\bH, H) := \tsum_{i=1}^{K} W(\bH_i, H_i).$$
Notice that $\bW$ is 1-strongly convex with respect to $\norm{H}_{\bW} := \sqrt{\tsum_{i=1}^{K} \norm{H_i}^2_W}$. 
Moreover, by fixing $\bH$ for the SSL algorithm and $H_0$ for the SD algorithm to be a uniform matrix with $1/K^2$ 
on every entry, the radii $\Omega_H^2 := \max_{H \in H_\delta} \bW(\bH, H)$ are bounded by $1/K$ for the Euclidean $\bW$ and $\log(K)$ for the entropy $\bW$. In the later analysis, to emphasize the relationship between $\Omega_H^2$ and $\Omega_p^2$, we define another constant $\Otp^2$ which has approximately the same range as $\Omega_P^2$:
\begin{itemize}
    \item[a)] When Euclidean $\bW$ is used, set $\Otp^2 := K \Omega_H^2$.
    \item[b)] When entropy $\bW$ is used, set $\Otp^2 := \Omega_H^2$.
\end{itemize}

Now if $\mathcal{H}_{\delta}$ denote the feasibility region of $H$, given by  $\{H \geq 0\ |\  \bar{p}_i = \tsum_{j=1}^{K} H_{i,j}\ \forall i, \innerF{H}{D} \leq \delta \}$,
the proximal update for $H$ using $\bW$ and the consequent update for the probability vector $q$ are:
\begin{equation}\label{eq:q_update}
\begin{split}
\Hnex := \argmax_{H \in \mathcal{H}_\delta} \inner{c}{\tsum_{i=1}^{K} H_i} - \mu_q \bW(\bH, H), \text{ and }  \qnex :=  \tsum_{i=1}^{K} \Hnex_i. 
\end{split}
\end{equation}
To differentiate it from the usual probability vector proximal update, we refer to \eqref{eq:q_update} as the $q$-update. 

By dualizing the $\innerF{H}{D} \leq \delta$ constraint, (\ref{eq:q_update}) becomes
\begin{align}\label{eq:q_refor}
\min_{\lambda \geq 0} \lambda \delta + \tsum_{i=1}^{K} \max_{H_i \geq 0, \inner{H_i}{e} = \bar{p}_i}
\inner{c}{H_i} - \mu_q W(H_{t,i}, H_i) - \lambda \inner{H_i}{D_i}.
\end{align}
Notice that for a fixed $\lam$, the inner maximization problem consists of $K$ independent simplex projection sub-problems, so it requires $\bigO(K^2)$ algebraic operations. 
If the bisection method is used to search for the optimal scalar $\lam^*$, we can find an $\ep_\lam$-suboptimal $\hat{\lambda}$ and $\hat{H}$ in roughly $\bigO(K^2 \log(1/\ep_\lam))$ algebraic operations, 
 a significant improvement over the original QP and the exponential cone problem. As shown in Table \ref{tb:q_update}, our numerical experiments 
written in MATLAB 2017a (with Mosek 8.1 as the QP/exponential cone solver) and tested on a Macbook Pro with 2.40GHz Intel Core i5 processor and 8GB of 1600MHz DDR3 memory demonstrate the significant performance improvement for the $q$-update.

\begin{table}[ht!]
\small
\caption{Typical Projection Time for $P_\delta$(Sec)}\label{tb:q_update}
\centering
\begin{tabular}{r | r  r| r r}
\toprule
\multicolumn{1}{c}{} &\multicolumn{2}{c}{Modified} & \multicolumn{2}{c}{Original} \\
\cmidrule(l){2-3} \cmidrule(l){4-5} 
\#Scenarios & Entropy & Euclidean &  Entropy &  Euclidean\\
\midrule
20& .0011 & .019 & 0.180 & 0.140  \\
100& .0028 & .030 & 0.538 & 0.228 \\
500& .047 & .16  & 16.15 & 6.615  \\
1000& .16 &.97 & 93.38 & 37.54 \\
5000 & 7.58 & 20.72 & Out.Mem & Out.Mem \\
\bottomrule
\end{tabular}
\end{table}

\subsection{Modified SD Method}

To use the more efficient $q$-update, we need to replace the update of $p_t$ in Line 6 of Algorithm~\ref{alg: sd} by 
\begin{equation}\label{eq:sd_H_update}
H_t := \argmax_{H \in \mathcal{H}_{\delta}} \inner{\tsum_{i=1}^{k} H_{i}}{\ftiltk} - \tauq \bW(H_{t-1} , H),
\end{equation}
and use $q_t := \tsum_{i=1}^{K} H_{t, i}$ in place of $p_t$ in all other parts of the algorithm. 
\vgap

Now we modify the arguments in Section 2 to establish the convergence properties of the modified SD method and 
suggest some stepsize choices. 
Recall that the analysis in Section 2 revolves around solving the saddle point problem $\min_{x \in X}\max_{(p, \bpi) \in P \times \bPi} \La(x, p, \bpi)$. Here we consider a modified saddle point problem associated with $H$ instead, i.e.,
\begin{equation}\label{eq:modified_La}
    \min_{x \in x} \max_{(H, \bpi) \in H_\delta \times \bPi} \{ \La(x, H, \bpi):= f_0(x) + \innerF{H}{(\innerT{x}{\bpi} - \bgstar(\bpi)) \bfe^\intercal}\},
\end{equation}
where $\bfe^\intercal \in \R^K$ is a row vector of ones, $[1, 1, \dots, 1]$. Similar to Proposition \ref{pr:duality} and \ref{pr:Q_func}, the non-negativity of $H$ implies the duality results between 
$\La$ and $f$. Then if the gap function $Q$ in Definition \ref{def:gap} is constructed from $\La(x, H, \bpi)$ in \eqref{eq:modified_La}, we have $N(f(\xbart) - f(\xstar)) \leq \max_{u_H, u_{\bpi} } \sumt Q\left[(x_t, H_t, \bpi_t), (\xstar, u_H, u_{\bpi})\right]$. Next, similar to Proposition \ref{pr:sd} and Theorem \ref{thm:sd}, the following convergence bounds of $\sumt Q(\zt, u)$ and the function value $f(\xbart)$ hold.

\begin{proposition}\label{pr:kan_sd}
If the non-negative stepsizes $\eta$, $\tau$ and $\sigma$ satisfy
\begin{equation}\label{eq:kan_stp_req}
 \eta \geq \tfrac{\cp^2 \Mt^2 \Mpi^2 K }{\tauq} + \tfrac{\Mt^2}{\sigma},
\end{equation}
then for any $u \in Z := X \times H \times \bPi$, we have
\begin{equation}\label{eq:kan_Q_convergence}
\tsum_{t=1}^{N} Q(z_t; u) \leq \sigma \inner{u_p}{U_k(\bpi_0, u_{\bpi})} + \tauq \bW( H_0, \uH) + \eta V(x_0, u_x).
\end{equation}
Moreover, for
$\bar{x}_N = \tsum_{t=1}^{N} \tfrac{\xt}{N}$ we have
   \begin{equation}\label{eq:kan_sd_g_convergence}
      f(\xbart) - f(\xstar) \leq \tfrac{\sigma \Opi^2 + \tauq \OH^2 + \eta \Ox^2}{N}, 
     \end{equation}
where $\OH^2 := \max_{H \in H_\delta} \bW(H_0, H)$.
\end{proposition}

\begin{proof}
We only need to modify the inequalities, \eqref{p_offset} and \eqref{ep_p},  related to the $p$-update. The modified Line 6 for the $q$-update in \eqref{eq:sd_H_update} implies that
\begin{align*}
&Q_H(\znext; u) \\
&\leq \innerF{u_H - \Hnext}{  (\innerT{\xnext - \xcur}{\pinext} - \innerT{\xcur - \xpre}{\picur}) \bfe^\intercal} + \tauq (\bW(\Hcur, \uH) - \bW(\Hnext, \uH) - \bW(\Hcur, \Hnext)) \\
& \leq \tauq (\bW(\Hcur, \uH) - \bW(\Hnext, \uH)) + [\innerF{\uH - \Hnext}{ \innerT{\xnext - \xcur}{\pinext}\bfe^\intercal} \\
& \quad - \innerF{\uH - \Hcur}{\innerT{\xcur - \xpre }{\picur}\bfe^\intercal}] + \epsilon_p(\Hnext), \stepcounter{equation}\tag{\theequation}\label{H_offset}
\end{align*}
where
 \begin{align*}
 \epsilon_p(\Hnext) &= \inner{\tsum_{i=1}^{K} H_{t+1,i} - H_{t, i}}{\innerT{\xcur - \xpre}{\picur}}  + \tauq \bW(\Hcur, \Hnext) \\
 & \stackrel{(a)}{\leq} \sqrt{K} \norm{\Hnext - \Hcur}_{\bW}  \norm{\xcur - \xpre}_2 \norm{[\norm{T_1 \bpi_{t,1}}_2, ... ,\norm{T_K \bpi_{t,K}}_2]}_{W^*} + \tauq \bW(\Hcur, \Hnext) \\
 &\leq \tfrac{1}{2 \tauq}  (Cp \Mt \Mpi)^2 K \norm{\xcur - \xpre}^2_2. \stepcounter{equation}\tag{\theequation}\label{ep_H}
\end{align*}
Note that (a) above follows from the algebraic fact $\norm{\tsum_{i=1}^{K} H_{t, i} - H_{t+1, i}}_W \leq \sqrt{K} \norm{H_t - H_{t+1}}_{\bW}$.
The rest of the proof for \eqref{eq:kan_Q_convergence} is the same as that for \eqref{eq:Q_convergence}.
Finally, (\ref{eq:kan_sd_g_convergence}) follows directly from (\ref{eq:kan_Q_convergence}) and the relation $N(f(\xbart) - f(\xstar)) \leq \max_{u_H, u_{\bpi} } \sumt Q(\zt, (\xstar, u_H, u_{\bpi}))$.
\end{proof}
\vgap 

Observe that the stepsize requirement (\ref{eq:kan_stp_req}) and the  convergence result (\ref{eq:kan_sd_g_convergence}) are exactly the same as their 
counterparts, (\ref{eq:stp_req}) and (\ref{thm:sd}) in Section 2, except for some constant factor. 
So we can apply a change of variables to reuse the stepsize policy developed in Theorem \ref{thm:sd}. More specifically, if
\begin{enumerate}
    \item[a)]  $\tiltau := \tauq$ and $\tilcp := \sqrt{K}$ for entropy $\bW$;
    \item[b)]   $\tiltau := \tauq /K$ and $ \tilcp := \sqrt{K}$ for Euclidean $\bW$,
\end{enumerate}
then we have $\tauq \OH^2 = \tiltau \Otp^2$ and $\cp^2 k / \tauq = \tilcp ^2 / \tiltau$.
 So the following convergence result and stepsize choice follow immediately from Proposition \ref{pr:kan_sd} and Theorem \ref{thm:sd}.

\begin{corollary}
For either the entropy $\bW$ or the Euclidean $\bW$, if the non-negative stepsizes satisfy
$
 \eta \geq \tfrac{\tilcp^2 \Mt^2 \Mpi^2}{\tiltau} + \tfrac{\Mt^2}{\sigma},
$
then we have
  \begin{equation*}\label{eq:kan_sd_g_convergence1}
    f(\xbart) - f(\xstar) \leq \tfrac{\sigma \Opi^2 + \tiltau \Otp^2 + \eta \Ox^2}{N}.
 \end{equation*}
In particular, if we choose
$\sigma := \Mt \tfrac{\Ox}{\Opi}, \tiltau := \Mt\Mpi \tilcp \tfrac{\Ox}{\Otp}, \text{ and } \eta := \Mt \Mpi \tilcp \tfrac{\Otp}{\Ox} + \Mt \tfrac{\Opi}{\Ox},$ then
$$f(\xbart) - f(\xstar) \leq \tfrac{2 \Ox \Mt}{N} (\Opi + \tilcp \Opi \Otp).$$
\end{corollary}

\subsection{Modified SSL Algorithm}

We replace the $p$-smoothing in \eqref{eq:smooth_F} with a $q$-smoothing to obtain a modified smooth approximation $\widetilde{F}_{\mu_q, \mupi}(x)$  given by  
 \begin{equation}\label{eq:kan_smooth_F}
\begin{aligned}
&\smoothG(x) := \max_{\pik \in \Pik} \inner{\pik}{\Tk x} - g_k^*(\pik) - \mupi U(0, \pik), & \text{($\bpi$ smoothing)} \\
&\widetilde{F}_{\mupi,\mu_q}(x) := \max_{H \in \mathcal{H}_{\delta}} \innerF{H}{\gmupi(x) \bfe^\intercal} - \mu_q \bW(\bar{H} , H) \text{ for some } \bar{H} \in \mathcal{H}_\delta.  & \text{($q$ smoothing)}
\end{aligned}
\end{equation}
To establish the $(\alpha, \beta)$-smooth approximation properties of $\widetilde{F}_{ \mupi, \mu_q}(x)$, we need the following \textit{domain smooth approximation} properties of the $q$-smoothing as a counterpart to Lemma \ref{lm:separate_smo}.b).

\begin{lemma}\label{lm:q_update_smo}
As a function of $\gmupi(x)$, $\widetilde{F}_{\mupi,\mu_q}$ is a $(K \norm{I}_{g,W}/\mu_q, \mu_q \OH^2)$-\textit{domain smooth approximation} of $F(\gmupi(x)) := \max_{H \in \mathcal{H}_{\delta}} \innerF{H}{\gmupi(x) \bfe^\intercal}$, where $\OH^2 := \max_{H \in H_\delta} \bW(\bar{H}, H).$
\end{lemma}

\begin{proof}
Let $\bg_1 := \gmupi(x_1)$ and $ \bg_2 := \gmupi(x_2)$ be given, and let $\Hnex_1$ and $ \Hnex_2$ be the corresponding maximizers in (\ref{eq:kan_smooth_F}).
Then we have
\begin{align*}
\norm{\qnex_1 - \qnex_2}^2_W &:= \norm{\tsum_{i=1}^K \hat{H}_{1,i} - \tsum_{i=1}^K \hat{H}_{2,i}}^2_W\stackrel{(a)}{\leq} K \norm{\Hnex_1 - \Hnex_2}^2_{\bW} \\
&\stackrel{(b)}{\leq} \tfrac{K}{\mu_q} \innerF{\Hnex_1 - \Hnex_2}{(\bg_1 - \bg_2)\bfe^\intercal} \\
&= \tfrac{K}{\mu_q} \inner{\qnex_1 - \qnex_2}{\bg_1 - \bg_2} \\
&\stackrel{(c)}{\leq} \tfrac{K}{\mu_q} \norm{\qnex_1 - \qnex_2}_W \norm{I}_{g,W} \norm{\bg_1 - \bg_2}_{g},
\end{align*}
where (a) follows from the algebraic fact that $\norm{\tsum_{i=1}^{K} H_i}^2_W \leq (\tsum_{i=1}^{K} \norm{H_i}_W)^2 \leq K \norm{H}^2_{\bW}$, (b) follows from Lemma \ref{lm:NesSmo}.b), and (c) follows from the definition of the operator norm $\norm{I}_{g,W}$. Dividing both sides by
$\norm{\qnex_1 - \qnex_2}_W$, we conclude that  $\widetilde{F}_{\mupi,\mu_q}$ is a Lipschitz smooth function of $\gmupi(x)$ with constant $K \norm{I}_{g,W}/\mu_q$. 
The  approximation gap follows from the definition of $\OH^2 := \max_{H \in H_\delta} \bW(\bar{H}, H).$
\end{proof}

The other parts needed to derive the smooth approximation properties of $\widetilde{F}_{\mupi,\mu_q}$, including the $\bpi$ smooth approximation properties and the Lipschitz continuity constant of $\smoothG$, are exactly the same as those in Section 3. Therefore Corollary \ref{cor:q_update_smo} below follows as an immediate consequence of Lemma \ref{lm:q_update_smo} and Proposition \ref{pr:smooth_cst}.
\begin{corollary}\label{cor:q_update_smo}
$\widetilde{F}_{\mupi,\mu_q}$ is a ($2\cp^2 K \Mt^2 \Mpi^2 / \mu_q + 2\Mt^2 / \mupi$, $\mu_q  \Omega^2_H + \mupi \Omega^2_{\bpi}$)-smooth approximation of $F$. 
\end{corollary}

Similar to the analysis of the modified SD algorithm, we can define a change of variables to simplify the above smooth approximation properties to 
the same form as that of $\smoothF$ in \eqref{eq:smooth_F} such that the SSL algorithm can be applied readily. More specifically,  we set
\begin{enumerate}
    \item[a)]  $\tilcp := \sqrt{k}$ and $\tilmup:= \mu_q$ for entropy $\bW$;
    \item[b)] $\tilcp := \sqrt{k}$ and $\tilmup := \mu_q /k$ for Euclidean $\bW$,
\end{enumerate}
such that $K\cp^2 / \mu_q = \tilcp^2 / \tilmup$ and $\mu_q \OH^2 = \tilmup \Otp^2$.
Then  $\widetilde{F}_{ \mupi, \mu_q}(x)$ is a ($2{\tilcp^2 } \Mt^2 \Mpi^2 / {\tilmup} + 2{\Mt^2} /{\mupi}$, $\tilmup  \Otp^2 + \mupi \Omega^2_{\bpi}$)-smooth approximation of $F$\footnote{The \textit{sequence smooth approximation} properties of $\widetilde{F}_{\mu_q, \mupi}$ can also be derived in a similar fashion.}, which is almost the same as $\smoothF$ being a ($2{\cp^2 } \Mt^2 \Mpi^2 / {\mup} + 2{\Mt^2} /{\mupi}$, $\mup  \Otp^2 + \mupi \Omega^2_{\bpi}$) smooth approximation (shown in Proposition \ref{pr:smooth_cst}). Since both the optimal smooth ratio (Lemma \ref{lm:opt_ratio}) and the SSL algorithm{\mytextquotesingle}s convergence analysis depend only on those smooth approximation properties, we conclude from Theorem~\ref{thm:ssl} that the SSL algorithm applied to $\widetilde{F}_{ \mupi, \mu_q}(x)$ has an iteration complexity of  $\bigO({\Ox \Mt \Opi (1 + 2 \tilcp \Otp)}/{\ep})$.

 \subsection{Iteration complexity}

Both the modified SD and the modified SSL algorithms have the same iteration complexity bound of $\bigO((1 + \tilcp \Otp)/\ep)$, i.e., $\bigO({\sqrt{K}}/{\ep})$ for Euclidean  $\bW$ and $\bigO({\sqrt{K \log K}}/{\ep})$ for entropy  $\bW$. It is worth noting that the extra $\sqrt{K}$ factor for entropy $\bW$ arises because the entropy radius scales sub-linearly, i.e. $\Omega_{\Delta/K  } = \Omega_{\Delta}/\sqrt{K}$ while the Euclidean radius scales linearly, i.e. $\Omega_{ \Delta/K  } = \Omega_{\Delta}/K$. Although the iteration complexity for the entropy $\bW$ is $\bigO(\sqrt{\log K})$ larger than that for the Euclidean $\bW$, 
 it is still preferable  in practice because each entropy projection is cheaper (shown in Table \ref{tb:q_update}).

\begin{table}[]
\small
    \centering
    
    \caption{Theoretical Performance Comparison for $\bigO(1 / \ep)$ Algorithms for Kantorovich $P$ Problem}
    \begin{threeparttable}
    
    \begin{tabular}{l|l | l}
        \toprule
         Algorithm & Iteration Complexity & Computation Required for $p$ Update \\
        \midrule
        Separable PDHG\cite{ChenXu2018Decomposition} & $\bigO(K/\ep)$ & Solving a $2K \times K^2$ QP \\
        \midrule
        Euclidean SD/SSL & $\bigO(\sqrt{K}/\ep)$ & Solving a $2K \times K^2$ QP \\
        Entropy SD/SSL & $\bigO(\sqrt{\log K}/\ep)$ & Solving a $2K \times K^2 $ Exponential Cone Program \\
        \midrule 
        Modified Euclidean SD/SSL & $\bigO(\sqrt{K}/\ep)$ & $\bigO(K^2 \log(1/\lam_\ep))$ Algebraic Computations \\
        Modified Entropy SD/SSL & $\bigO(\sqrt{K \log K}/\ep)$ & $\bigO(K^2 \log(1/\lam_\ep))$ Algebraic Computations \\
        \bottomrule
    \end{tabular}
    
    \begin{tablenotes}
        \item[1] We set $\lam_\ep$ to the machine precision.
    \end{tablenotes}
    \end{threeparttable}
    
\end{table}

\section{Numerical Studies}
We use distributionally robust two-stage linear programs to demonstrate the empirical performance
of our algorithms . 

Firstly, we test our algorithms by measuring their average performance on some randomly generated instances of a synthetic problem. 
We consider the following capacity installation problem of an electricity utility company.
\begin{align}\label{eq:n_test}
\begin{split}
 \min_{x \in R^n}\quad & c^\intercal x + \max_{p \in P } \tsum_{k=1}^{K} p_k g_k(\Tk x)  \\
 & s.t.\ \    0 \leq x_i \leq U \ \forall i,\\
 \text{where}&\quad g_k(\Tk x) := \min_{y_k \in \R^m_+} y_k^\intercal e_k \\
& \quad \ \ \ \ \quad \quad \quad \ \ s.t.\    \bR y_k \geq d_k - \Tk x.
\end{split}
\end{align}
The company is planning for the capacities of $n$ technologies, $x \in \R^n$, to be installed for the coming year, with a unit cost vector $c \in R^n$. 
Moreover, being the sole provider of electricity in the region, it has to satisfy all demands in different periods of the year, $d \in \R^m$, using a combination of power generated by those installed capacities, with an availability factor of  $T \in [0, 1]^{m \times n}$, and power purchased from the outside grid at a unit cost of $e \in R^m$. The stochastic parameters $e, d,$ and $T$ are unavailable at the planning time, so the company needs to find either a data-driven or risk-averse solution. In our experiments, we set $m=20,\ n=40$ and generate random instances in the following fashion.
\begin{enumerate}
    \item  $c$ generated entry-wise IID from Unif[0.5 1].
    \item $e_k \in R^m$ generated entry-wise IID from Unif[2,4].
    \item $d_k \in R^m$ generated entry-wise IID from Unif[50, 100].
    \item $\Tk \in R^{m \times n}$ generated entry-wise IID from Unif[0.5 1].
    \item  $\bR = I_{m, m}$ is the simple complete recourse matrix.
\end{enumerate}
{\color{blue}
Since $\bR$ is the identity matrix, the scenario sub-problems are simple. They admit closed-form solutions for a given $x$, and  each $\Pik$ is a box in \eqref{prob}, so the $\bpi$-proximal update is also simple. However, in reformulation \eqref{eq:Hu_reformulation} of \cite{ChenXu2018Decomposition}, 
the projection of $(x_k, v_k)$ onto a non-smooth functional constrained feasibility set, $\{(x_k, v_k) |v_k \geq g_k(\Tk x_k), x_k \in X\}$,  is  more difficult; we have to solve a QP. 

We also verify our results on a real-world test instance, namely the telecommunication network expansion problem with uncertain demands, SSN(50) \cite{SSN}. However, rather than the original expected total unfilled demand, we use some risk-averse function of the total unfilled demands as the objective function. Moreover, since we have to solve for the flow over the network for each demand scenario, the scenario sub-problems are more difficult ($R$ is not the simple identity matrix and $\Pi(k)$s are not boxes). So we have to use LP and QP solvers for them and their proximal updates.}

\subsection{Implementation Details}
The numerical experiments are implemented in  MATLAB 2017b with Mosek 8.1 as the optimization solver and are tested on an Alienware Desktop with 4.20GHz Intel Core i7 processor and 16GB of 2400MHz DDR4 memory. The $x$ proximal updates and level set projection problems are solved using Mosek QP and the $\bpi$ proximal updates are solved using closed-form solutions for the synthetic problem and using Mosek QP for the network expansion problem. The $p$ proximal updates are solved to machine accuracy  using a binary search for the Lagrange multipliers associated with the coupling constraints in $P$. Their computation complexities  are listed in Table \ref{tab:p_projection}.

Given a test instance,  SD and PDHG are first fine-tuned by selecting among a few parameter choices the one achieving the smallest objective value in 100 iterations,
 $f(\bar{x}_{100})$(see Table \ref{tb:Param} for these parameter choices). Next, the fine-tuned SD and PDHG and the 
 parameter-free SSL are used to solve the instance.
 We record the number of iterations and the wall clock time required  for these algorithms to achieve a relative optimality gap of $\ep \in \{10\%, 1\%, 0.1\%\}$, i.e., $f(x_t) - f_* \leq \ep f_*$. If the target accuracy is not reached after 2,000 seconds, we record both the number of iterations and the time as NA.  To obtain an estimate of the true objective $f^*$,  we use the parameter-free SSL algorithm and terminate only when the absolute gap between the lower and upper bound decreases to $1e^{-3}$. 
 
\begin{table}[]
\small
    \centering
    \caption{\# Algebraic Operations Required for $p$ proximal update}
    
    \begin{tabular}{l|l|l}
        \toprule
         Distance Function $W$ \& Ambiguity Set & Constraints in $P$ &  \# Algebraic Operations  \\
         \midrule
         Euclidean or Entropy  $W$ \& Simplex & Box + One Linear  & $O(K \log(1/\ep))$ \\
         \midrule
         Euclidean or Entropy $W$ \& AVaR & Box + One Linear  & $O(K \log(1/\ep))$ \\
         \midrule
         Euclidean $W$ \& Modifed $X^2$ & Box + One Linear  & $O(K \log^2(1/\ep))$ \\
         &  + One Quadratic  & \\
         \midrule
         Modified Entropy or Euclidean $\bW$ \& Kantorovich & See Section 4 & $O(K^2 \log(1/\ep))$ \\
         \bottomrule
    \end{tabular}

    \label{tab:p_projection}
\end{table}

\begin{table}[ht!]
\small
\begin{threeparttable}
\caption{Parameter Selections}
\label{tb:Param}
\centering
\begin{tabular}{l| l| l}
\toprule
Algorithm & \#  & Step-sizes \\
\toprule
&& Over-relaxation parameter $\rho = 2$.\\
PDHG & 27 & $\eta \in \{10^{-3}, 10^{-2}, ..., 10^2, 10^3\}$.\\
&& $\tau \in \{10^{-2}, 10^{-1}, 1\} / \eta$.\\
\midrule
&& $(\sigma, \tau, \eta) \in \{.1 \bar{\sigma}, \bar{\sigma}\} \times \{.1 \bar{\tau}, \bar{\tau}\} \times \{.1 \bar{\eta}, \bar{\eta}\}$,\\
SD & 16 & where $\bar{\sigma}, \bar{\tau}$ are $\bar{\eta}$ calculated using the stepsize choice in Theorem \ref{thm:sd}\\
&& with conservative estimates of $\Ox, \Op, \Opi$ and with $\{.1\Mpi, \Mpi\}$.\\
\midrule 
SSL & 1 & $\lam_0 = 2^{-6}$, $\Omega^2_{p,0} = W(\bar{p}, p_0)$, $M^2_{\pi, 0} = 2 \max_{k} U_k(\bar{\bpi}, \bpi_0)$,\\
& & where $\bar{p}, \bpi_0$ are the maximizers for $x_0$ in \eqref{prob}.\\
\bottomrule
\end{tabular}
\begin{tablenotes}
\small
\item Note that the parameter estimation for the PDHG algorithm is difficult because both the primal and the dual feasibility region are unbounded. 
\end{tablenotes}
\end{threeparttable}
\end{table}

\subsection{Synthetic Problem: Probability Simplex Ambiguity Set}
Notice that both SD and SSL have the same iteration complexity bound of $\bigO({1+ \cp \Op}/{\ep})$.
So to best illustrate how they scale with $K$, we conduct experiments on the probability simplex, which has the largest $\Op$. We make a few remarks about the result obtained in Table \ref{tb:num_scen}. 

\begin{enumerate}
  \item In general when the number of scenarios is large, both  SSL and  SD  show significant improvement over PDHG  in both computation time and iteration number. This is consistent with numerical experiments in \cite{ChenXu2018Decomposition}, where a toy example (with $m=3$ and $n=2)$ takes a significant amount of time even for a small number of scenarios, $K\leq 200$.  Besides, SSL  seems to outperform SD in finding solutions with high accuracy. 

  \item Dependence on accuracy $\ep$: both  SD and PDHG  match the theoretical complexity guarantee of $\bigO({1}/{\ep})$. In contrast, SSL enjoys a linear rate of convergence in practice, i.e., $\bigO(c^{t})$, for some $c < 1$. Such a behavior is often observed for bundle level methods \cite{BenNem2005NERM,Lan15Bundle,Nem1995Bundle}, but there is no rigorous theoretical explanation to the best of our knowledge.

  \item Dependence on the number of scenarios $K$: both the computation time and the number of iterations required for  PDHG  increase quickly with $K$. However,  the numbers of iterations required for entropy SD and SSL are nearly scenario independent. In fact, they seem to decrease slightly with increasing $K$. One plausible explanation is that more scenarios make $f$ smoother, thus our accelerated algorithms might converge faster. However for Euclidean $ W$, the number of iterations required for  SD increases for large $K$ while that for  SSL  stays the same. 

  \item Per iteration computation time:  the per iteration computation time of  PDHG is larger than that of SD and SSL. Moreover,  
  the projection of $x$ onto a level set in SSL is more expensive than the simple $x$-proximal update in  SD. So when the number of scenarios is small ($K = 50 \sim2000$) such that the level set projection dominates computation cost, 
  SD seems to be faster than SSL for finding $1\%,10\%$-suboptimal solutions, even though its numbers of iterations required are larger. 
  However, when the number of scenarios is large and the $\bpi$ projection dominates the computation cost, SSL is faster.  
\end{enumerate}

\subsection{Synthetic Problem: Risk-Averse AVaR Ambiguity Set}
Given the empirical probability vector $\bar{p}$, we use the following reformulation in  \cite{ShapiroAhmed2004Minimax} of AVaR risk measure in our experiments.
\begin{align*}
AVaR_{1 - \alpha}[g_1(x), ..., g_K(x)] = &\max_{p \geq 0} \ \inner{p}{\bg(x)} & \\
& s.t\ \ \   \tsum_{k=1}^{K} p_k = 1 & \\
&\quad\ \ \ \ 0 \leq p_k \leq \tfrac{1}{\alpha} \bar{p}_k. & 
\end{align*}
Observe that results shown in Table \ref{tb:cvar} are consistent with our findings in Subsection 5.2. In addition, both the iteration numbers and the computation times for all algorithms increase slightly in the $97.5\%$ AVaR quantile case because of the larger $\Op$.

\subsection{Synthetic Problem: Modified $X^2$ Ambiguity Set}
The modified $X^2$ in \cite{ChenXu2018Decomposition} is defined as 
$$
\mathcal{P}_r :=\left\{p \in \mathbb{R}_{+}^{k} : \norm{p-[\tfrac{1}{K}, .... ,\tfrac{1}{K}]}_2^{2} \leq r, \tsum_{i=1}^{K} p_{i}=1\right\}.
$$
Since the entropy projection onto a quadratically constrained $\mathcal{P}_r$ is difficult, we conduct experiments using only Euclidean $W$. The obtained result in Table \ref{tb:X2} is consistent with our previous findings.

\subsection{Synthetic Problem: Kantorovich Ball Ambiguity Set}

We test the modified SD and modified SSL algorithms developed in Section 4 for the more challenging  Kantorovich ball. The results are presented in Table \ref{tb:kan}. When $K= 200$, the computation time in each iteration due to the  Euclidean $p$-update in PDHG is 0.2 second, while that for the entropy $q$-update in both SD and SSL algorithms is 0.02 second. When $K$ is larger, the saving from the $q$-update is even more significant. 

\subsection{Real-world instance: SSN(50)}
We conduct tests on the SSN(50) problem in  \cite{SSN} with all the above-mentioned ambiguity sets. The obtained results in Table \ref{tb:ssn} show that our SD and SSL algorithms significantly outperform the PDHG algorithm in computation time. Notice that the number of iterations of SD is comparable to that of PDHG and its saving derives mainly from easier $\bpi$ proximal updates (as compared to the joint $(x, v_k, y_k)$ epigraph projection in PDHG). For problems with a large number of scenarios, we expect our scenario-independent algorithms to have a more significant advantage over PDHG in iteration number as well. 
{\color{blue}
\subsection{Comparison with the Benders Decomposition Algorithm} Finally, we compare the SSL algorithm with another frequently used cutting plane method, the Benders decomposition \cite{bertsimas1997introduction}. Our implementation considers the following master problem,
\begin{align}
        \min_{x \in X, \Psi, v_k}\ & c^\intercal x + \Psi \notag\\
        s.t.\ & \Psi \geq \phi(v_1,...,v_K) \label{eq:bender_p}\\
        & v_k \geq g_k(\Tk x)\quad  \forall k \in [K] \label{eq:bender_scen}.
\end{align}
In each iteration, the algorithm first computes $(\xt, \Psi_t, \{v_{k,t}\})$ by minimizing a master model, and then  adds  optimality cuts for the risk function in \eqref{eq:bender_p} and for the scenario cost functions in \eqref{eq:bender_scen} to the master model. 

We test these algorithms on the synthetic problem with both 50 and 1000 scenarios and on the SSN(50) problem, and the results are listed in Table \ref{tb:syn_50_bender}, \ref{tb:syn_bender_1000}, and \ref{tb:ssn_bender} respectively. It is clear that the Benders decomposition algorithm outperforms our SSL algorithm when the number of scenarios is small and the scenario sub-problems are simple (Table \ref{tb:syn_50_bender}). However, when either the number of scenarios is large (Table \ref{tb:syn_bender_1000}) or the scenario sub-problems are difficult (Table \ref{tb:ssn_bender}), the SSL algorithm converges much faster. 
}

To sum up, our experiments demonstrate that the proposed SD and SSL algorithms show significant performance improvement over the PDHG algorithm, especially for problems with a large number of scenarios. Between SSL and SD, SSL seems to be a better choice because it does not require any parameter tuning and it has a linear rate of convergence in practice. 
However, the SD algorithm is simpler to implement and may have some performance advantages over SSL for small problems with a low accuracy requirement. Moreover, the flexibility to choose a Bregman distance  appropriate for the $P$ geometry has a significant influence on the per iteration computation time, which is evident in the Kantorovich ball experiment.

\begin{table}[ht!]
\tiny
\begin{threeparttable}
\caption{Synthetic Problem: Simplex \\ {\small mean number of iterations and time(sec) to reach desired relative optimality gap}}
\label{tb:num_scen}
\centering
\begin{tabular}{l| r| r| r r| r r}
\toprule

\#Scenarios&  Gap & 
PDHG & SD Euclid & SD Entropy & SSL Euclid & SSL Entropy   \\
\midrule
 & 10\% & 
  333, 11.3s  & 
     268, 0.18s & 200, 0.13s & 74, 0.28s  &    74, 0.26s   \\

20 & 1\% & 
  3940, 146s  & 
    4060, 3.04s & 2510, 1.57s & 153, 0.60s &    142, 0.58s   \\

 & 0.1\% & 
  NA, NA & 
     NA, NA & 23600, 16.1s & 260, 1.09s&    246, 1.02s \\

  \midrule
 & 10\% & 
  NA, NA & 
     62, 0.35s  & 44, 0.27s & 94, 1.12s&    94, 1.14s \\

 200 & 1\% & 
  NA, NA & 
   602, 3.26s & 476, 2.65s & 181, 2.43s&    181, 2.45s \\

 & 0.1\% & 
  NA, NA & 
     6010, 32.1s  & 4810, 26.0s & 307, 4.32s&    311, 4.50s \\

  \midrule
 & 10\% & 
  NA, NA & 
     48, 1.02s & 44, 0.94s & 101, 7.49s&    91, 6.75s \\

1000 & 1\% & 
  NA, NA & 
     471, 10.4s & 394, 8.62s & 184, 15.0s &    177, 14.3s   \\

 & 0.1\% & 
  NA, NA & 
     4710, 102s & 3840, 83.3s & 293, 24.7s  &    291, 24.7s   \\

\midrule
 & 10\% & 
  NA, NA & 
  123, 46.4s &   34, 14.6s & 86, 64.0s  &    94, 76.1s   \\

20000 & 1\% & 
  NA, NA & 
  1210, 461s &   220, 92.8s & 168, 139s  &    178, 160s   \\

 & 0.1\% & 
  NA, NA & 
  NA, NA &   2020, 821s & 285, 248s  &    285, 274s   \\

\bottomrule
\end{tabular}
\end{threeparttable}
\end{table}

\begin{table}[ht]
\tiny
\begin{minipage}{0.49\linewidth}
\begin{threeparttable}
\caption{Synthetic Problem: AVaR}
\label{tb:cvar}
\centering
\begin{tabular}{l| r| r| r |r|}
\toprule
\#Scenarios&  Gap & 
PDHG & SD Entropy & SSL Entropy   \\
\midrule
\multicolumn{5}{c}{95\% AVaR quantile }\\
\midrule
 & 10\% & 
  1170, 104s & 
  111, 0.19s &   63, 0.59s   \\

50 & 1\% & 
  12100, 1040s & 
  1340, 2.61s &   115, 0.90s  \\

 & 0.1\% & 
  NA, NA & 
  13700, 23.6s &   225, 1.98s  \\
\midrule
 & 10\% & 
  NA, NA & 
  30,0.15s  &   57, 0.49s   \\

  200 & 1\% & 
  NA, NA & 
  362, 2.08s  &   120, 1.22s   \\

 & 0.1\% & 
  NA, NA & 
  3570, 19.8s &   233, 2.55s   \\

\midrule
 & 10\% & 
  NA ,NA & 
  9, 0.15s  &   40, 0.94s    \\

  1000 & 1\% & 
  NA, NA & 
  122, 2.70s  &   69, 1.94s    \\

 & 0.1\% & 
  NA, NA & 
  1180, 25.8s &   120, 3.90s    \\
  \midrule
\multicolumn{5}{c}{97.5\% AVaR quantile }\\
\midrule
 & 10\% & 
  1390, 108s & 
  118, 0.16s &   70, 0.32s   \\

50 & 1\% & 
  14400, 1140s  & 
  1410, 1.98s &   154, 0.76s  \\

 & 0.1\% & 
  NA, NA & 
  14600, 21.3s &   290, 1.49s  \\
\midrule
 & 10\% & 
  4140, 1250s  & 
  34, 0.29s  &   60, 0.48s   \\

  200 & 1\% & 
  NA, NA & 
  410, 2.28s  &   139, 1.36s   \\

 & 0.1\% & 
  NA, NA & 
  4090, 21.6s &   259, 2.76s   \\

\midrule
 & 10\% & 
  NA, NA & 
  24, 0.51s  &   47, 1.11s    \\

  1000 & 1\% & 
  NA, NA & 
  205, 4.34s &   88, 2.60s    \\

 & 0.1\% & 
  NA, NA & 
  2030, 43.2s &   176, 5.85s    \\
  \bottomrule
\end{tabular}
\end{threeparttable}
\end{minipage}
\begin{minipage}{0.49\linewidth}
\tiny
\begin{threeparttable}
\caption{Synthetic Problem: Modified $X^2$}
\label{tb:X2}
\centering
\begin{tabular}{|l| r| r| r |r}
\toprule

\#Scenarios&  Gap & 
PDHG & SD Euclid & SSL Euclid  \\
\midrule
\multicolumn{5}{c}{$r=0.01$}\\
\midrule
 & 10\%& 514, 40.2s & 112, 0.18s & 43, 0.22s   \\

50& 1\%& 4120, 330s & 1530, 1.96s & 82, 0.60s  \\

 & 0.1\%&   NA, NA & 15200, 18.0s & 157, 1.32s  \\
\midrule
 & 10\%& 1040, 314s & 26, 0.19s & 40, 0.38s  \\

200& 1\%& 4570, 1370s & 268, 1.94s & 70, 0.92s  \\

 & 0.1\%&   NA, NA & 2110, 11.8s & 116, 1.77s  \\

\midrule
 & 10\%&   NA, NA & 20, 0.47s & 43, 1.50s    \\

1000& 1\%&   NA, NA & 95, 2.64s & 67, 2.97s   \\

 & 0.1\%&   NA, NA & 330, 8.80s & 99, 4.89s    \\
  \midrule
\multicolumn{5}{c}{$r=0.1$}\\
\midrule
 & 10\%& 996, 76.1s & 163, 0.19s & 59, 0.37s   \\

50& 1\%& 9970, 804s & 2340, 2.69s & 121, 0.94s  \\

 & 0.1\%&   NA, NA& 20400, 21.6s & 245, 2.14s  \\
\midrule
 & 10\%& 3390, 997s & 68, 0.29s & 62, 0.76s   \\

200& 1\%&   NA, NA& 860, 3.81s & 128, 1.91s   \\

 & 0.1\%&   NA, NA& 8920, 39.0s & 238, 3.81s   \\

\midrule
 & 10\%& NA, NA& 70, 1.48s & 64, 2.83s    \\

1000& 1\%& NA, NA& 717, 15.3s & 126, 6.37s    \\

 & 0.1\%& NA, NA& 7160, 151s & 230, 12.2s    \\







  \bottomrule
\end{tabular}
\end{threeparttable}
\end{minipage}
\end{table}

\begin{table}[ht!]
\tiny
\begin{threeparttable}
\caption{Synthetic Problem: Kantorovich Ball}
\label{tb:kan}
\centering
\begin{tabular}{l| r| r| r |r}
\toprule

\#Scenarios&  Gap & 
PDHG & Modified SD Entropy & Modified SSL Entropy   \\
\midrule
\multicolumn{5}{c}{$\delta=0.01 \times \text{median distance in $D$}$}\\
\midrule
 & 10\%& 247, 27.0s & 89, 0.18s & 43, 0.33s   \\

50& 1\%& 1560, 181s & 846, 1.96s & 70, 0.66s  \\

 & 0.1\%&   NA, NA & 8170, 19.2s & 130, 1.41s  \\
\midrule
 & 10\%& 498, 297s & 16, 0.28s & 38, 1.69s   \\

200& 1\%& 1590, 883s & 163, 3.05s & 55, 2.90s   \\

 & 0.1\%&   NA, NA & 1180, 23.0s & 90, 5.74s   \\

\midrule
 & 10\%&   NA, NA & 16, 13.6s & 36, 76.5s    \\

1000& 1\%&   NA, NA & 131, 106s & 50, 127s    \\

 & 0.1\%&   NA, NA & 602, 443s & 73, 221s    \\
  \midrule
\multicolumn{5}{c}{$\delta=0.1 \times \text{median distance in $D$}$}\\
\midrule
 & 10\%& 420, 49.9s & 93, 0.23s & 55, 0.59s   \\

50& 1\%& 2940, 363s & 847, 2.10s & 101, 1.30s  \\

 & 0.1\%&   NA, NA & 8270, 17.9s & 175, 2.45s  \\
\midrule
 & 10\%& 756, 557s & 20, 0.43s & 52, 3.79s   \\

200& 1\%&   NA, NA & 111, 2.62s & 87, 7.24s   \\

 & 0.1\%&   NA, NA & 564, 13.0s & 136, 12.1s   \\

\midrule
 & 10\%&   NA, NA & 20, 20.6s & 51, 171s    \\

1000& 1\%&   NA, NA & 96, 96.0s & 78, 297s    \\

 & 0.1\%&   NA, NA & 358, 334s & 117, 485s    \\
   \midrule







\end{tabular}
\end{threeparttable}
\end{table}

\begin{table}[ht!]
\tiny
\begin{threeparttable}
\caption{SSN(50)}
\label{tb:ssn}
\centering
\begin{tabular}{l| r| r| r r| r r}

\toprule
Ambiguity Set &  Gap & 
PDHG\tnote{(1)}\ \ \  & SD Euclid & SD Entropy & SSL Euclid & SSL Entropy   \\
\midrule
 Simplex  & 10\% & 
  420, 873s  & 
     360, 63.3s & 309, 54.4s & 142, 40.0s  &    107, 31.3s   \\

& 1\% & 
  NA, NA  & 
    996, 176s & 631, 111s & 195, 55.2s &    187, 54.0s   \\

  \midrule
  \midrule
 AVaR & 10\% & 
  1116, 2310s & 
     366, 64.4s  & 233, 40.8s & 69, 19.3s&    135, 40.3s \\

 95\% & 1\% & 
  NA, NA & 
   1016, 179.5s & 640, 113s & 148, 41.6s&    211, 61.8s \\

  \midrule
 AVaR & 10\% & 
  420, 880s & 
     355, 62.9s & 310, 55.1s & 94, 26.5s&    117, 36.0s \\

97.5\% & 1\% & 
  NA, NA & 
     818, 146s & 634, 113s & 152, 43.0s &    190, 57.0s   \\


\midrule
\midrule
 $X^2$ & 10\% & 
  167, 337s & 
  61, 10.7s &    & 144, 44.3s  &       \\

$r=0.01$ & 1\% & 
  NA, NA & 
  129, 22.8s &    & 181, 55.0s  &      \\

\midrule
 $X^2$ & 10\% & 
  582, 1210s & 
  96, 17.3s &    & 128, 40.6s  &       \\

$r=0.1$ & 1\% & 
  NA, NA & 
  477, 84.9s &    & 192, 60.3s  &      \\

  \midrule
  \midrule
 Kantorovich\tnote{(2)} & 10\% & 
  179, 373s & 
     285, 119s  & 784, 137s & 119, 73.1s&    90, 25.0s \\

 $\delta=0.01$ & 1\% & 
  837, 1750s & 
   768, 306s & 2133, 382s & 177, 110s&    135, 43.1s \\

  \midrule
 Kantorovich & 10\% & 
  210, 445s & 
     235, 107s & 1105, 198s & 108, 66.7s&    104, 29.2s \\

$\delta=0.1$ & 1\% & 
  NA, NA & 
     533, 234s & 3281, 595s & 167, 105s &    179, 51.0s   \\


\bottomrule

\end{tabular}

\begin{tablenotes}
\tiny
\item[(1)] Both SD and PDHG use the best iterate encountered (instead of the ergodic average) to measure the optimality gap for faster convergence.
\item[(2)] Modified Euclidean and entropy projections used for SD and SSL.
\end{tablenotes}

\end{threeparttable}

\end{table}

\begin{table}[ht!]
\centering
\makebox[0pt][c]{\parbox{\textwidth}{%
    \begin{minipage}[b]{0.40\hsize}\centering
    \tiny
    \begin{threeparttable}
    \captionsetup{font=small}
    \caption{Synthetic(50)}
    \label{tb:syn_50_bender}
    \centering
        {\color{blue}
        \begin{tabular}{l| r| r| r|}

        \toprule
        Ambiguity Set &  Gap & 
        SSL\tnote{(1)}\ \ \  & Benders   \\
        \midrule
        Simplex  & 10\% & 
        101, 0.48s  & 
        37, 0.36s   \\

        & 1\% & 
        202, 1.08s  & 
        38, 0.38s    \\

        \midrule
        \midrule
        AVaR & 10\% & 
        79, 0.35s & 
        14, 0.08s   \\

        95\% & 1\% & 
        129, 0.64s & 
        15, 0.09s  \\

        \midrule
        AVaR & 10\% & 
        64, 0.27s & 
        22, 0.17s \\

        97.5\% & 1\% & 
        150, 0.77s & 
        23, 0.19s   \\


         \midrule
         \midrule
         $X^2$ & 10\% & 
         46, 0.19s & 
         7, 0.03s      \\

         $r=0.01$ & 1\% & 
         96, 0.52s & 
         10, 0.07s      \\

         \midrule
         $X^2$ & 10\% & 
         49, 0.22s & 
         10, 0.07s       \\

         $r=0.1$ & 1\% & 
         103, 0.57s & 
         11, 0.08s       \\

         \midrule
         \midrule
         Kantorovich & 10\% & 
         47, 0.31s & 
         8, 0.08s  \\

         $\delta=0.01$ & 1\% & 
         71, 0.67s & 
         19, 0.22s \\

         \midrule
         Kantorovich & 10\% & 
         56, 0.49s & 
         5, 0.04s  \\

         $\delta=0.1$ & 1\% & 
         107, 1.10s & 
         7, 0.07s    \\


         \bottomrule

         \end{tabular}

         \begin{tablenotes}
         \tiny
        \item[(1)] SSL uses entropy projection for all but the $X^2$ ambiguity set.
        
        \end{tablenotes}
        }
    \end{threeparttable}
    \end{minipage}
\begin{minipage}[b]{0.3\hsize}\centering
      \tiny
    \begin{threeparttable}
    \captionsetup{font=small}
    \caption{Synthetic(1000)}
    \label{tb:syn_bender_1000}
    \centering
        {\color{blue}

        \begin{tabular}{|r| r|}

        \toprule
       \ \ \ \ \ \ \ \ \ \hspace{0.5cm} SSL\tnote{(1)}\ \ \ & \ \ \ \ \  \  Benders    \\
        \midrule
        103, 3.04s &
        101, 58.0s   
          \\

       184, 5.76s  &
        106, 60.5s   
          \\

        \midrule
        \midrule

         48, 1.15s &
        12, 6.23s   
          \\

       88, 2.44s  &
        19, 7.96s   
          \\

        \midrule
        
        48, 1.16s & 
        17, 7.02s \\

        92, 2.66s & 
        26, 8.89s   \\


         \midrule
         \midrule
         \rule{0pt}{5.25pt}   
        
        35, 0.79s & 
        8, 5.06s \\

        64, 2.04s & 
        22, 8.36s   \\

         \midrule

         \rule{0pt}{5.25pt}49, 1.37s & 
         26, 10.3s       \\

         95, 3.19s & 
         37, 12.4s       \\

         \midrule
         \midrule
  
         42, 60.0s & 
         5, 16.6s       \\

         52, 96.8s & 
         6, 19.0s       \\

         \midrule
         50, 101s & 
         5, 15.8s  \\

         76, 187s & 
         16, 40.0s    \\

         \bottomrule

         \end{tabular}

        }
    \end{threeparttable}
    \end{minipage}
    \begin{minipage}[b]{0.25\hsize}\centering
      \tiny
    \begin{threeparttable}
    \captionsetup{font=small}
    \caption{SSN(50)}
    \label{tb:ssn_bender}
    \centering
        {\color{blue}
        \begin{tabular}{|r| r|}

        \toprule
        SSL\tnote{(1)}\ \ \ & Benders   \\
        \midrule
        
        106, 30.4s  & 
        NA, NA   \\

        185, 52.4s  & 
        NA, NA    \\

        \midrule
        \midrule
        
        135, 39.4s & 
        NA, NA   \\

        211, 60.4s & 
        NA, NA  \\

        \midrule
        
        117, 35.1s & 
        NA, NA \\

        174, 51.1s & 
        NA, NA   \\


         \midrule
         \midrule
         \rule{0pt}{5.25pt}144, 44.1s  &
         975, 578s \\

         179, 54.2s  &
         NA, NA    \\

         \midrule

         \rule{0pt}{5.25pt}128, 39.7s & 
         NA, NA       \\

         187, 57.4s & 
         NA, NA       \\

         \midrule
         \midrule
  
         90, 24.7s & 
         NA, NA  \\

         145, 40.2s & 
         NA, NA \\

         \midrule
         104, 29.0s & 
         NA, NA  \\

         159, 44.7s & 
         NA, NA    \\

         \bottomrule

         \end{tabular}
        }
    \end{threeparttable}
    \end{minipage}
    }
}
\end{table}



\section{Conclusion}
{\interlinepenalty 10000
This paper considers the distributionally robust two-stage stochastic convex program with a discrete scenario support. To handle the large number of scenarios and the non-smooth second stage cost function, we propose a sequential maximization reformulation of the problem and develop a simple SD algorithm and a parameter-free SSL algorithm. Both algorithms are able to achieve a nearly scenario independent iteration complexity of $\bigO(\sqrt{\log K} /\ep)$. Moreover, for the difficult but important Kantorovich ball, we develop a modification of our algorithms to avoid the expensive projection onto $P$ at the price of $\bigO{(\sqrt{K})}$ times more iterations. The empirical performance of our algorithms is demonstrated  by encouraging numerical experiment results. 

Moreover, since the subproblems in the SD and SSL algorithms are assumed to be solved exactly,
an interesting question is how our algorithms would perform when using  quick but not-so-accurate solutions. Indeed, this type of question can inspire the development of new algorithms (e.g, gradient sliding methods \cite{Lan2016Gradient}), but it would require substantial modifications to both our algorithms and their analysis. So we will consider it in our future research.
}

\bibliographystyle{siam} 
\bibliography{pdd.bib}
\end{document}